\documentclass[12pt,centertags,oneside]{amsart}
\usepackage{amsmath,amstext,amsthm,amscd,typearea}
\usepackage{amssymb}
\usepackage{a4wide}
\usepackage[mathscr]{eucal}
\usepackage{mathrsfs}
\usepackage{typearea}
\usepackage{charter}
\usepackage{pdfsync}
\usepackage{url}
\usepackage{xcolor}

\usepackage[a4paper,width=16.2cm,top=3cm,bottom=3cm]{geometry}

\numberwithin{equation}{section}

\newtheorem{theorem}{Theorem}[section]
\newtheorem{definition}[theorem]{Definition}
\newtheorem{proposition}[theorem]{Proposition}
\newtheorem{corollary}[theorem]{Corollary}
\newtheorem{lemma}[theorem]{Lemma}
\newtheorem{remark}[theorem]{Remark}

\newcommand{\cali}[1]{\mathscr{#1}}

\newcommand{\lov}{{\rm lov}}
\newcommand{\supp}{{\rm supp}}

\newcommand{\dist}{\mathop{\mathrm{dist}}\nolimits}

\renewcommand{\Im}{\mathop{\mathrm{Im}}\nolimits}
\newcommand{\vol}{\mathop{\mathrm{vol}}}

\newcommand{\ddc}{\text{\normalfont dd}^c}
\newcommand{\dc}{\text{\normalfont d}^c}
\newcommand{\Ta}{\text{\normalfont T}}

\def\d{\operatorname{d}}
\def\Ker{\operatorname{Ker}}

\newcommand{\id}{{\rm id}}

\newcommand{\codim}{{\rm codim\ \!}}

\newcommand{\C}{\mathbb{C}}
\newcommand{\D}{\mathbb{D}}
\newcommand{\N}{\mathbb{N}}

\newcommand{\R}{\mathbb{R}}
\renewcommand\P{\mathbb{P}}

\title[]{Densities of currents and complex dynamics}

\author{Duc-Viet Vu}
\address{University of Cologne, Mathematical Institute, Germany}
\email{vuduc@math.uni-koeln.de}

\date{\today}
\begin{document}

\begin{abstract} We extend the Dinh-Sibony notion of  densities of currents to the setting where the ambient manifold  is not necessarily K\"ahler and study the intersection of analytic sets from the point of view of densities of  currents.  As an application, we introduce the notion of \emph{exotic periodic points} of a meromorphic self-map.  We then establish  \emph{the expected  asymptotic} for the sum of the  number of isolated periodic points and the number of exotic periodic points for holomorphic self-maps with a simple action on the cohomology groups on a compact K\"ahler manifold.   We also show that the algebraic entropy of meromorphic self-maps of compact complex surfaces is a finite  bi-meromorphic invariant. 
\end{abstract}

\maketitle

\medskip

\noindent
{\bf Classification AMS 2010}:  32U40, 32H50, 37F05. 

\medskip


\tableofcontents

\section{Introduction} \label{introduction}

A fundamental problem in the pluripotential theory and complex geometry is to define in a reasonable way the intersection of two closed positive currents.  Although, the intersection of currents of bi-degree $(1,1)$ is well understood (see \cite{Demailly_ag,Fornaess_Sibony,Bedford_Taylor_82}), the case of currents of higher bi-degree still remains challenging.  

A  recent remarkable progress in this research direction is the theory of densities of currents on K\"ahler manifolds given by  Dinh-Sibony \cite{Dinh_Sibony_density}, which  generalizes  the theory of  super-potentials and the classical theory of  intersection of currents of bi-degree $(1,1)$ mentioned above, see \cite{DNV,Viet_Lucas} for details.  This theory has deep applications to complex dynamics and foliations. We refer the reader to \cite{DS_unique_ergodicity,DNS_foliation,VA-lyapunov,DNV,DNT_equi,DVT_growth_periodic,lucas_foliation} for details.
   
The first aim of this paper is to develop the theory of densities of currents  in the setting where the ambient manifold is not necessarily K\"ahler and study the \emph{excess} intersection of analytic sets from the point of view of density currents. The second aim of this paper is to apply this study to complex dynamics. We present below such two applications. 

Let $X$ be a compact complex  manifold of dimension $k$. Let $f$ be a  dominant meromorphic self-map  of $X.$  
Let $\omega$ be a  strictly positive Hermitian $(1,1)$-form on $X.$  For $0 \le q \le k,$ put
$$d_q(f):= \limsup_{n \to \infty}\bigg( \int_X  (f^n)^* \omega^q \wedge  \omega^{k-q}\bigg)^{1/n}, \quad h_a(f):= \max_{0 \le q \le k}\{\log d_q(f)\}.$$
We will write $d_q$ for $d_q(f)$ if no confusion arises.   We can see easily that $d_j$ is independent of the choice of $\omega$ for every $j.$  The number $d_0$ is always equal to  $1$ and $d_k$  is equal to  the topological degree of $f.$ When $f$ is holomorphic, these numbers $d_j$ are finite because  the differential of $f$ is of uniformly bounded norm on $X.$ We call $d_q$ the $q^{th}$ \emph{dynamical degree} of $f$ for $0 \le q \le k$ and $h_a(f)$ \emph{the algebraic entropy}  of $f.$  

When $X$ is K\"ahler, the numbers $d_j, h_a(f)$ are crucial \emph{finite} bi-meromorphic invariants of $f;$ see  \cite{DS_upperbound,Ds_upperbound_mero,DS_regula}.  A \emph{ (isolated) periodic point of period $n$} of $f$ is by definition a (isolated) point in the intersection of the graph $\Gamma_n$ of $f^n$ and the diagonal $\Delta$ of $X^2.$    Let $P_n$ be the number of isolated periodic points of $f$ of period $n$ counted with multiplicity  and $h_t(f)$ the topological entropy of $f.$  

We will  introduce  the notion of \emph{ exotic  periodic points} and their multiplicities, see Definition \ref{def-exotic-periodic} below. Denote  by $\tilde{P}_n$  the sum of $P_n$ and the number of exotic periodic points of period $n$ counted with multiplicity.   
Our first main result, which is a direct consequence of Theorem \ref{the_equidistrG} in Section \ref{sec_mapdominant},  gives the expected asymptotic for $\tilde{P}_n$ in the K\"ahler case.

\begin{theorem} \label{the_equidistr}  We have 
\begin{align}\label{ine_upperboundPnrefined}
\tilde{P}_n  =  e^{n h_a(f)}+ o(e^{n h_a(f)})
\end{align}
if one of the following situations occurs:

$(i)$  $X$ is a compact K\"ahler manifold and $f$ is a holomorphic self-map of $X$ with a simple action on the cohomology groups.

$(ii)$ $X$ is a compact K\"ahler surface and $f$ is an algebraically stable (dominant) meromorphic self-map of $X$ with a minor (small) topological degree.
\end{theorem}

We refer to Section \ref{sec_mapdominant} for the definition of maps with a simple action on the cohomology groups.    Examples of automorphisms with a simple action on cohomology groups in dimension $>2$ can be found in \cite{Oguiso-Truong}. In the situation $(i)$ above, (\ref{ine_upperboundPnrefined})  still holds for a good class of  holomorphic correspondences, see Theorem \ref{the_equidistrG} for details.  Dinh-Nguy\^en-Truong \cite{DVT_growth_periodic} raised the question of whether the equality
$$P_n=  e^{n h_a(f)}+ o(e^{n h_a(f)})$$ 
holds for a large class of $f.$   So in view of (\ref{ine_upperboundPnrefined}), in order to solve the last question, one is led to expect that the number of exotic periodic points is at most $o(e^{n h_a(f)})$ at least for the two cases mentioned in Theorem \ref{the_equidistr}. The techniques developed in this paper are however not sufficient to obtain it.
   
We notice that  a related result for algebraically stable bimeromorphic maps on K\"ahler surfaces was given by  Dinh-Nguy\^en-Truong based on the Saito's local index function, see  \cite[Th. 1.4]{DVT_growth_periodic}. It is an interesting question to investigate the relation between the notion of multiplicity in the sense of Saito's local index function in \cite{Saito,KT_area} and that given in this paper.  
  
In the two situations in Theorem \ref{the_equidistr}, the upper bound $$P_n  \le e^{n h_a(f)}+ o(e^{n h_a(f)})$$  was proved in \cite{DNV,DVT_growth_periodic}.  Our contribution here is that the gap between $P_n$ and its expected asymptotic is the number of exotic periodic points of period $n.$ This will be obtained as a direct consequence of our study of the excess intersection of analytic sets. 

If $X$ is $\P^k$ and $f$ is a non-invertible holomorphic endomorphism, then every periodic point is isolated, \emph{i.e,} $\Gamma_n$ intersects $\Delta$ properly. In this case,  (\ref{ine_upperboundPnrefined}) is obtained easily by using B\'ezout's theorem, see \cite{DS_book}. In general, the fact that $\Gamma_n$  doesn't intersects $\Delta$ properly  is a main difficulty in estimating $P_n.$  This explains the need to have a good intersection theory for analytic sets or currents where the dimension excess phenomenon occurs.   

Here is a list of other cases where the asymptotic of $P_n$ was known (and the equidistribution of periodic points was also proved except for the last case):

$\bullet$  $X= \P^k$ and $f$ is a H\'enon-type map by   Dinh-Sibony \cite{DS_periodic_Henon},

$\bullet $  Meromorphic self-maps with dominant topological degree on compact K\"ahler manifolds by  Dinh-Nguy\^en-Truong   \cite{DNT_equi}, see also Dinh-Sibony \cite{DS_allurepolynom},

$\bullet$  Algebraically stable meromorphic self-maps on compact K\"ahler surfaces satisfying certain additional conditions by   Diller-Dujardin-Guedj \cite{Diller-Guedj-Dujardin} and Dinh-Nguy\^en-Truong \cite{DVT_growth_periodic}, see also \cite{Cantat}, 

$\bullet$ Area-preserving birational maps of projective surfaces with certain additional conditions by  Iwasaki-Uehara \cite{KT_area}.

Our second main result concerns the algebraic entropy and $P_n$ in  the non-K\"ahler case.

\begin{theorem} \label{th_main2phay}  Let $X$ be a compact complex surface and $f$ a meromorphic self-map of $X$.  Then the algebraic entropy $h_a(f)$ of $f$ is a finite bi-meromorphic invariant of $f$ and 
\begin{align}\label{ine_upperboundhf}
h_t(f) \le h_a(f)<\infty,
\end{align}
and 
\begin{align}\label{ine_upperboundPn}
\limsup_{n \to \infty} \frac{1}{n} \log  P_n \le h_a(f).
\end{align}
\end{theorem}

The new point in the above theorem is the case where $X$ is \emph{non-K\"ahler.} We refer to \cite{Gromov_entropy,Vu_nonkahler_topo_degree} for examples of self-maps on non-K\"ahler manifolds.  We don't know whether $d_q$  are finite for general  $X$ of dimension $>2.$    When $X$ is a K\"ahler manifold of arbitrary dimension, (\ref{ine_upperboundPn}) is proved by Dinh-Nguy\^en-Truong \cite[Th. 1.1]{DVT_growth_periodic}.   The upper bound (\ref{ine_upperboundhf}) was proved by Gromov \cite{Gromov_entropy} for holomorphic self-maps of compact K\"ahler manifolds and by  Dinh-Sibony  \cite{DS_upperbound,Ds_upperbound_mero} for meromorphic self-correspondences of compact K\"ahler manifolds.  The proofs in these last papers use, in an essential way, a regularisation theorem for closed positive currents in \cite{DS_superpotential} which is not  available in the non-K\"ahler case.

Theorem \ref{th_main2phay} is a direct consequence of Theorem \ref{th_main2phayG} in Section 4 which is in turn deduced from our study of the intersection of analytic subsets on non-K\"ahler manifolds.

The paper is organized as follows. In Section 2, we present a generalization of the theory of tangent currents  to non-K\"ahler manifolds.  In Section \ref{sec_analy}, we use the last theory to study the  intersection of analytic sets.    Theorems \ref{the_equidistr} and \ref{th_main2phay} are proved in Section \ref{sec_mapdominant}.  More applications to complex dynamics for non-K\"ahler manifolds are also given in Section \ref{sec_mapdominant}.  
\\

\noindent
\textbf{Acknowledgments.} The author would like to express his gratitude to Tien-Cuong Dinh and Taeyong Ahn  for fruitful discussions on the paper \cite{Dinh_Sibony_density}. He also thanks Tuyen Trung Truong and Vi\^et-Anh Nguy\^en for stimulating comments.  This research is supported by a postdoctoral fellowship of Alexander von Humboldt Foundation.

\section{Tangent currents} \label{sec_tangent}

Let $X$ be a complex manifold of dimension $k$ and $V$ smooth complex submanifold of $X$ of dimension $l.$  Let $T$ be a closed positive $(p,p)$-current on $X,$ where $0 \le p \le k.$ We assume that $T$ \emph{has no mass on $V.$} By Federer's support theorem \cite{Federer},  every closed positive current can be decomposed into the sum of a closed positive current having no mass on $V$ and one on $V.$ Hence,  the  hypothesis on $T$ in fact makes no restriction in our study. Let $\supp T$ be the support of $T.$   Let $[V]$ be the current of integration along $V.$  

 Denote by $\pi: E\to V$ the normal bundle of $V$ in $X$ and $\overline E:= \P(E \oplus \C)$ the projective compactification of $E.$ The hypersurface at infinity $H_\infty: = \overline E \backslash E$ of $\overline E$  is naturally isomorphic to $\P(E)$ as fiber bundles over $V.$ We also have a canonical projection $\pi_\infty: E \backslash V  \to H_\infty. $ 

Let $U$ be an open subset of $X$ with $U \cap V \not = \varnothing.$  Let $\tau$ be  a smooth diffeomorphism  from $U$ to an open neighborhood of $V\cap U$ in $E$ which is identity on $V\cap U$ such that  the restriction of its differential $d\tau$ to $E|_{V \cap U}$ is identity.  Such a map is called \emph{an admissible map}. When $U$ is a small enough tubular neighborhood of $V,$ there always exists an admissible map $\tau$ by \cite[Le. 4.2]{Dinh_Sibony_density}. In general, $\tau$ is not holomorphic.  When $U$ is a small enough local chart, we can choose an admissible holomorphic map by using suitable holomorphic coordinates on $U$.   

For $\lambda \in \C^*,$ let $A_\lambda: E \to E$ be the multiplication by $\lambda$ on fibers of $E.$ Consider the family of closed currents $(A_\lambda)_* \tau_* T$ on $E|_{V \cap U}$ parameterized by $\lambda \in \C^*.$

\begin{definition} \cite{Dinh_Sibony_density,Viet_Lucas} A tangent current  $T_\infty$ of $T$ along $V$ is a closed positive current on $E$ such that there are a sequence $(\lambda_n) \subset \C^*$ converging to $\infty$ and a collection of admissible holomorphic maps $\tau_j: U_j \to E$ for $j \in J$ satisfying the following two properties. 

$(i)$ $$V \subset \cup_{j \in J} U_j,$$

$(ii)$ 
$$T_\infty:= \lim_{n \to \infty} (A_{\lambda_n})_* (\tau_j)_* T$$
on $\pi^{-1}(U_j \cap V)$ for every $j \in J.$
\end{definition}

When $X$ is K\"ahler and $\supp T \cap V$ is compact, the above definition of tangent currents agrees with that given in  \cite{Dinh_Sibony_density} and it is  proved there  that tangent currents always exist and are independent of the choices of $\tau_j$. This crucial fact also holds in our setting, see Lemma \ref{le_globaladmiss} below.      By this reason, the sequence $(\lambda_n)$ is called \emph{the defining sequence} of $T_\infty.$   Before introducing a weaker assumption (Hypothesis \textbf{(H)} below) guaranteeing the existence of tangent currents,  we will give some notations and auxiliary results.  

Following \cite{Dinh_Sibony_density}, a bi-Lipschitz map $\tilde{\tau}$ from $U$ to an open neighborhood of $U\cap V$ in $E$ is said to be \emph{almost-admissible} if   $\tilde{\tau}|_{U \cap V}= \id,$ $\tilde{\tau}$ is smooth outside $V$ and on every local chart $\big(W,x=(x',x'')\big)$ near $V\cap U$ with $V\cap W=\{x''=0\}$ then 
$$\tilde{\tau}(x)= \big(x'+ O^*(|x''|), x''+ O^*(|x''|^2)\big),$$
and 
$$ \quad d \tilde{\tau}(x)=\big(dx'+ O^{**}(|x''|), dx''+ O^{**}(|x''|^2)\big),$$
where for every positive integer $m,$ $O^*(|x''|^m)$ means a function which is continuous outside $V$ and is equal to $O(|x''|^m)$ as $x'' \to 0;$ $O^{**}(|x''|^m)$ means the sum of $1$-forms with $O^*(|x''|^m)$ coefficients and a combination of $dx'', d \overline x''$ with $O^*(|x''|^{m-1}).$ 

Note that $\tilde{\tau}_* T$ is well-defined as a closed current on $\tilde{\tau}(U) \backslash (U \cap V)$ which is of locally finite mass near $U \cap V$ because $\tilde{\tau}$ is bi-Lipschitz and smooth outside $V.$ We  extend $\tilde{\tau}_* T$ to be  a current of order $0$ on $\tilde{\tau}(U)$ by putting $\tilde{\tau}_*T:= 0$ on $U \cap V.$ Although  $\tilde{\tau}_* T$ is actually closed (see \cite[4.1.14]{Federer}), we will not need that fact in the sequel. When $\tilde{\tau}$ is smooth, it is clear that $\tilde{\tau}_* T$ is the usual pushforward of $T$ by $\tilde{\tau}$ because $T$ has no mass on $V.$      We will need to use both admissible and  almost-admissible maps. 

\begin{lemma} \label{le_globaladmiss} (\cite{Viet_Lucas}) Let $T_\infty$ be a tangent current of $T$ along $V$ with the defining sequence $(\lambda_n)_{n\in \N}.$  Then for any almost-admissible map $\tilde{\tau}: U \to E$, we have 
$$T_\infty= \lim_{n \to \infty} (A_{\lambda_n})_* \tilde{\tau}_* T$$
on $\pi^{-1}(U \cap V).$ 
\end{lemma}

\proof We follow closely the arguments from \cite{Viet_Lucas}.  Let $\tau_j, U_j$ with $j \in J$ be as above. Fix a $j \in J.$ Without loss of generality, we can suppose that $U=U_j$ is a local chart. Put $\tau:= \tau_j.$ Let $(x', x'')$ be local coordinates on $U$ for which $V \cap U= \{x''=0\}.$  Identify $E$ with $(V \cap U) \times \C^{k-l}$, recall here that $l= \dim V$. Since both $\tau$ and $\tilde{\tau}$ are almost-admissible, we have 
\begin{equation} \label{eq:rho-tau}
\left( \tau - \tilde{\tau} \right) (x',x'') = \big(O^*(|x''|), O^*(|x''|^2)\big), \quad \left( d\tau - d\tilde{\tau} \right) (x',x'') = \big(O^{**}(|x''|), O^{**}(|x''|^2)\big), 
\end{equation}
as $x''\to 0$.   Let $\Phi$  be a smooth form with compact support in $\pi^{-1}(U \cap V) \subset E$ and denote 
$$\Phi_\lambda:= (A_\lambda)^* \Phi.$$  
 Notice that $|x''| \lesssim |\lambda|^{-1}$ on  the support of $\tau^* \Phi_\lambda -  \tilde{\tau}^* \Phi_\lambda$. Using this fact and (\ref{eq:rho-tau}), on $U \backslash V,$ we have 
$$\tau^* \Phi_\lambda -  \tilde{\tau}^* \Phi_\lambda   = \frac{1}{|\lambda|} \tau^* A_\lambda^* \Psi_\lambda $$ 
where $\Psi_\lambda$ are  forms on $\pi^{-1}(U \cap V) \backslash V$ supported in  a fixed compact subset of $E$ and the coefficients of $\Psi_\lambda$ are uniformly bounded on $\pi^{-1}(U \cap V)$ in $\lambda$. Let $\Omega$ be a positive form  with compact support on $\pi^{-1}(U\cap V)$ such that $\Psi_\lambda \leq \Omega$ on $\pi^{-1}(U\cap V) \backslash V$ for every $\lambda$. Since  $ (A_\lambda)_* \tau_*  T$ is positive and $T$ has no mass on $V$, we have  
\begin{align*}
\big|\langle  T, \tau^* \Phi_\lambda -  \tilde{\tau}^* \Phi_\lambda \rangle \big| &= \big|\langle  T, \mathbf{1}_{U \backslash V}(\tau^* \Phi_\lambda -  \tilde{\tau}^* \Phi_\lambda) \rangle \big| \\
&\le   |\lambda|^{-1} \big|\langle (A_\lambda)_* \tau_*  T, \Psi_\lambda \rangle \big| \leq |\lambda|^{-1}\big|\langle (A_\lambda)_* \tau_*  T, \Omega \rangle \big|.
\end{align*}
By the hypothesis,  $\lim_{\lambda_n \to \infty}(A_{\lambda_n})_* \tau _* T= T_\infty.$ Thus  the mass of $(A_{\lambda_n})_* \tau _* T$ on compact sets is bounded uniformly in $\lambda_n$. This gives 
\begin{align} \label{ine_taunaga}
\big|\langle T, \tau^* \Phi_{\lambda_n} -  \tilde{\tau}^* \Phi_{\lambda_n} \rangle \big| \le C  |\lambda_n|^{-1}.
\end{align}
for some constant $C$ independent of $n.$ Hence $\lim_{n\to \infty} (A_{\lambda_n})_* \tilde{\tau}_* T= T_\infty.$  The proof is finished. 
\endproof

For two closed positive currents $T_1, T_2$ on $X.$  Consider the tensor current $T_1 \otimes T_2$ on $X \times X.$   A \emph{density current} associated to $T_1, T_2$ is a tangent current of $T_1 \otimes T_2$ along the diagonal $\Delta$ of $X \times X.$  Consider a particular case where  $T_1:= T$ and $T_2:= [V].$ We will show that a density current associated to $T, [V]$ corresponds naturally to a tangent current of $T$ along $V.$ 

 Observe that we have  natural identifications $\Ta (X^2) \approx \Ta X \times \Ta X$ between vector bundles, where $\Ta X$ is the tangent bundle of $X$ and $\Delta \approx X$.  Since $V \subset X\approx \Delta,$ there is a canonical inclusion $\imath$ from  $\Ta V$ to $(\Ta X \times \{0\})|_\Delta$ which is a subbundle of $\Ta (X^2)|_{\Delta}.$   Let $F$  be the image of $\imath(\Ta V)$ in the normal bundle $E_\Delta= \Ta (X^2)/ \Ta \Delta.$ Put $\Delta_V:= \{(x,x) \in X^2:x \in V\}.$  Let $E_{\Delta, V}$ be the restriction of $E_{\Delta}$ to $\Delta_V.$  Observe that $F$  is a subbundle of $E_{\Delta,V}$ of rank $l$ and the natural map 
$$\Psi: E_{\Delta, V}/F \to E=\Ta X/ \Ta V$$
 is an isomorphism. Let $p_V: E_{\Delta,V} \to E_{\Delta,V}/F $ be the natural projection. 

\begin{lemma} (\cite[Le. 5.4]{Dinh_Sibony_density}) \label{le_deltatangecurr}  If $T_\infty$ is  a tangent current of $T$ along $V,$ then the current $p_V^* \Psi^* T_\infty$ is a tangent current of $T \otimes [V]$ along $\Delta.$  Conversely, every tangent current of $T\otimes [V]$ along $\Delta$ can be written as $p_V^* \Psi^* T_\infty$ for some tangent current $T_\infty$ of $T$ along $V.$ 
\end{lemma}

\proof  Firstly notice that every tangent current of $T \otimes [V]$ along $\Delta$ is supported on  $\pi_{\Delta}^{-1}\big((\supp T \times V) \cap \Delta\big),$ where $\pi_\Delta$ is the natural projection from $E_\Delta$ to $\Delta.$  Hence, a such current is supported on $E_{\Delta_V}.$ 

Let $T_\infty$ be  a tangent current of $T$ along $V$ and $(\lambda_n)$ its defining sequence.  Consider a local chart $(U,x)$ of $X$ with $U= U' \times U''$ and  $x=(x',x'')$ so that $V\cap U$ is given by the equation $x''=0$ and $0 \in U.$ We then obtain an induced local chart $U \times U$ with coordinates $(x,y)$ on $X \times X$ with $x=(x',x'')$ and $y=(y',y'').$ The diagonal $\Delta$ is given by the equation $x=y$ on $U \times U.$ Put $z=(z',z''):= x-y,$ $z'=x'-y',$ $z''= x''-y''.$   Thus, for an  open subset $U_1=U_1' \times U''_1$ of $U$ small enough containing $0$,  $\big(U_1^2,(x,z)\big)$ is also a local chart  on $X^2$ with $\Delta= \{z=0\}.$ 

Using the local coordinates $(x,z),$ we  identify the tangent bundle of $X^2$ on $U_1^2$ with $U_1^2 \times \C^{2k}$ and $E_\Delta$ with $U_1 \times \C^k$ which is embedded in $U_1^2 \times \C^{2k}$ as $U_1 \times \{0\} \times \C^k.$   Similarly we also identify $\Ta X$ on $U$ with $U \times \C^k$ and $E$ with $U' \times \C^{k-l}.$ With these identifications, we see that  
$$E_{\Delta,V} \approx U' \times \C^k, \quad F \approx U' \times \C^{l}\times \{0\}.$$
It follows that  
$$p_V:  U' \times \C^k \to U' \times \{0\} \times \C^{k-l}, \quad \Psi:  U' \times \{0\} \times \C^{k-l} \to U' \times \C^{k-l}.$$
We also have that  the identity maps $\id_U: U \to U$ and $\id_{U_1^2}: U_1^2 \to U_1^2$ are (local) 
holomorphic admissible maps for $V, \Delta$ on $X, X^2$ respectively. By definition of $T_\infty,$ we 
get $T_\infty= \lim_{n \to \infty} (A_{\lambda_n})_* T$  on $U' \times \C^{k-l}.$ Thus 
\begin{align}\label{eq_pullTinfty}
p_V^* \Psi^* T_\infty=   \lim_{n \to \infty} p_V^* \Psi^* (A_{\lambda_n})_* T=\lim_{n \to \infty} (A_{\lambda_n})_* p_V^* \Psi^* T
\end{align}
because $(A_\lambda)_*= A_{\lambda^{-1}}^*$ and $A_{\lambda^{-1}}$ commutes with the vector bundle maps $p_V, \Psi.$ We now prove that $(A_{\lambda_n})_* (T \otimes [V])$ is convergent on $U_1^2 \times \C^{k-l}.$  Let $\Phi= \Phi_0(x')\wedge \Phi_1(x'')  \wedge \Phi_2(z') \wedge \Phi_3(z'')$ be a smooth form with compact support in $U_1 \times \C^k.$ The set of forms $\Phi$ is dense in $\cali{C}^\infty$-topology in  the space of smooth forms with compact support.   We consider first the case where $\Phi_1$ is a function in $x''.$   Without loss of generality, we can assume $\Phi_1(0)=1.$  We have 
\begin{align} \label{eq_tinhpVPsiT2}
\big \langle (A_{\lambda_n})_* (T \otimes [V]), \Phi \big \rangle &= \big \langle T \otimes [V], \Phi_0(x') \wedge \Phi_1(x'') \wedge \Phi_2(\lambda_n z') \wedge \Phi_3(\lambda_n z'') \big \rangle\\
\nonumber
&= \big \langle T(x), \Phi_0(x') \wedge \Phi_1(x'')  \wedge \Phi_3(\lambda_n x'') \wedge \int_{y' \in V} \Phi_2\big(\lambda_n(x'-y')\big) \big \rangle\\
\nonumber
&= \big \langle T(x), \Phi_0(x') \wedge \Phi_3(\lambda_n x'') \wedge \int_{(y',0) \in V} \Phi_2\big(\lambda_n(x'-y')\big) \big \rangle+ O(|\lambda_n|^{-1})
\end{align}
because $x'' \to 0$ as $\lambda_n \to \infty$ and $(A_{\lambda_n})_*T$ is of uniformly 
bounded mass on compact subsets of $U' \times \C^{k-l}.$ Observe that
$$ \int_{(y',0) \in V} \Phi_2\big(\lambda_n(x'-y')\big)= \int_{z' \in x' -U'_1} \Phi_2(\lambda_n z')= \int_{z' \in \lambda_n^{-1}(x' -U'_1)} \Phi_2(z') = \int_{\C^l} \Phi_2$$
for every $x'$ in a fixed compact set  if  $n$ big enough because $\supp \Phi_2 \Subset \C^l$ which is contained in   $ \lambda_n^{-1}(x' -U'_1)$ if $|\lambda_n|$ is  big.   This together with (\ref{eq_tinhpVPsiT2}) implies 
\begin{align} \label{eq_tinhpVPsiT3}
\big \langle (A_{\lambda_n})_* (T \otimes [V]), \Phi \big \rangle &=  \big \langle T(x), \Phi_0(x') \wedge \Phi_3(\lambda_n x'') \wedge [\int_{\C^l} \Phi_2] \big \rangle+ o_{n\to \infty}(1)\\
\nonumber
&=\big \langle (A_{\lambda_n})_*T, \Phi_0(x') \wedge \Phi_3(x'') \wedge [\int_{\C^l} \Phi_2] \big \rangle+ o_{n\to \infty}(1) .   
\end{align}
Notice that 
$$\Phi|_{E_{\Delta,V}}= \Phi_0(x')\wedge  \Phi_2(z') \wedge \Phi_3(z'')$$
because $\Phi_1(0)=1.$ Using this and (\ref{eq_pullTinfty}) gives
\begin{align} \label{eq_tinhpVPsiT}
\big \langle p_V^* \Psi^* T_\infty, \Phi|_{E_{\Delta,V}} \big \rangle &= \big \langle T_\infty, \Psi_* (p_V)_* (\Phi|_{E_{\Delta,V}}) \big \rangle\\
\nonumber
&= \big \langle T_\infty, \Psi_* (p_V)_* [\Phi_0(x') \wedge \Phi_2(z') \wedge \Phi_3(z'')]\big \rangle\\
\nonumber
& = \big \langle T_\infty,  \Phi_0(x') \wedge \Phi_3(x'')[\int_{\C^l} \Phi_2] \big \rangle. 
\end{align}
Comparing (\ref{eq_tinhpVPsiT}) and (\ref{eq_tinhpVPsiT3}) gives $\lim_{n \to \infty} (A_{\lambda_n})_* (T \otimes [V])=p_V^* \Psi^* T_\infty.$ Consider now $\Phi_1$ is a form of degree $\ge 1.$
Then $\Phi|_{E_{\Delta,V}}=0.$ It follows that $\langle p_V^* \Psi^* T_\infty, \Phi \rangle =0.$ On the other hand,  we also see from (\ref{eq_tinhpVPsiT2})-(\ref{eq_tinhpVPsiT3}) that $\big \langle (A_{\lambda_n})_* (T \otimes [V]), \Phi \big \rangle \to 0$ as $n \to \infty.$ Consequently,  $\lim_{n \to \infty} (A_{\lambda_n})_* (T \otimes [V])=p_V^* \Psi^* T_\infty$ holds in the both cases.  

We now assume that $T'_\infty:=\lim_{n \to \infty} (A_{\lambda_n})_* (T \otimes [V])$ exists. Then, by choosing $\Phi_1(x'') \equiv 1$ in the above defining formula of $\Phi$ and using (\ref{eq_tinhpVPsiT2})-(\ref{eq_tinhpVPsiT}), we obtain that 
  $(A_{\lambda_n})_* T$  is of uniformly mass on compact subsets of $U' \times \C^{k-l}.$ Hence, there is a subsequence $(\lambda'_n)$ of $(\lambda_n)$ for which $(A_{\lambda'_n})_* T \to T_\infty$ for some $T_\infty.$ 

The first part of the proof then implies that $T'_\infty= p_V^* \Psi^* T_\infty.$ Hence,  $T_\infty$ is the unique limit current of the sequence $(A_{\lambda_n})_* T$. In other words, $\lim_{n \to \infty}(A_{\lambda_n})_* T= T_\infty$ and $T'_\infty= p_V^* \Psi^* T_\infty.$   This finishes the proof.
\endproof

Let $\sigma: \widehat X \to X$ be the blowup along $V$ of $X$ and $\widehat V:= \sigma^{-1}(V)$ the exceptional hypersurface. Recall that $\widehat V$ is naturally biholomorphic to $\P(E).$  Let $\sigma_E: \overline{\widehat E} \to \overline{E}$ the blowup along $V$ of $\overline E.$ The restriction of $\sigma_E$ to $\widehat E:= \sigma_E^{-1}(E)$ is the blowup along $V$ of $E.$     The projection $\pi$ induces naturally a vector bundle projection $\pi_{\widehat E}$ from $\widehat E$ to $\sigma_E^{-1}(V).$ The last map  can be extended to a projection $\pi_{\overline{\widehat E}}$  from  $\overline{\widehat E}$ to $\sigma_E^{-1}(V).$  
 The vector bundle $\pi_{\overline{\widehat E}}: \widehat E\to \sigma_E^{-1}(V)$ is naturally identified with the normal bundle of $\widehat V$ in $\widehat X.$  Hence we can identify $\sigma_E^{-1}(V)$ with $\widehat V$ and use $\widehat E$ as the normal bundle of $\widehat V$ in $\widehat X.$  

 Given any smooth admissible map $\tau: U \to E,$ by \cite[Le. 4.3]{Dinh_Sibony_density}, we can lift $\tau$ to a bi-Lipschitz almost-admissible map $\widehat{\tau}$ with 
\begin{align} \label{eq_lifttau}
\sigma_E  \circ \widehat \tau = \tau \circ \sigma_E.
\end{align}
Observe that the hypersurface at infinity $\widehat H_\infty$ of  $\overline{\widehat E}$ is biholomorphic to  that of $\overline E$ via $\sigma_E$. We  use $\widehat{\pi}_\infty$ to denote the natural projection from $\overline{\widehat E}\backslash \widehat V$ to $\widehat H_\infty.$ Since the rank of $\overline{\widehat E}$ over  $\widehat V$ is $1,$ we can extend $\widehat{\pi}_\infty$ to a projection from $\overline{\widehat E}$ to $\widehat H_\infty.$  Let $\widehat T$ be the pull-back of $T$ on $\widehat X  \backslash \widehat V$ by $\sigma|_{\widehat X \backslash \widehat V}.$ We assume from now on the following.\\
 
 \noindent
\textbf{(H): }    \emph{$\widehat T$ has locally finite mass near $\widehat V$ and there are countably many holomorphic admissible maps $\widehat \tau_j: \widehat U_j \to \widehat E$ with $j \in J$ such that  $\widehat V \subset \cup_{j \in J} \widehat U_j,$   $(A_\lambda)_* (\widehat \tau_j)_* \widehat T$ is of uniformly bounded mass on  compact subsets of $\pi_{\widehat E}^{-1}(\widehat U_j \cap \widehat U)$ as $|\lambda| \to \infty$ for every $j \in J.$}\\

\noindent
We will prove in  Theorem \ref{th_Hhold} at the end of this section   that the last assumption is satisfied  if $\supp T \cap V$ is compact  and there exists a Hermitian form $\omega$ on $X$ with $\ddc \omega^j=0$ on $V$ for $1 \le j \le k-p-1.$ This generalizes   the criteria given in  \cite[Th. 4.6, Le. 3.12]{Dinh_Sibony_density} for K\"ahler manifold $X$ where the above form $\omega$ is a K\"ahler form on $X.$   Another interesting  case where \textbf{(H)} is satisfied is when $T$ is a current of integration along an analytic subset of $X,$ see Theorem \ref{th_V1alongV} in the next section. 

Note that since $\widehat T$ has locally finite mass near $\widehat V,$ it can be extended trivially through $\widehat V$ to be a closed positive current on $\widehat X.$   We still use $A_\lambda$ to denote the multiplication by $\lambda \in \C^*$ in fibers of $\widehat E$ or $\overline{\widehat E}.$

By a diagonal argument, Hypothesis \textbf{(H)} ensures the existence of a tangent current  $\widehat T_\infty$  to $\widehat T$ along $\widehat V$ associated with  a sequence $(\lambda_n) \subset \C \to \infty.$ 
The following result is essentially contained in \cite{Dinh_Sibony_density}.

\begin{proposition} For any smooth admissible map $\tau: U \to E,$  the mass $(A_{\lambda})_* \tau_* T$ on compact subsets of $E|_{V \cap U}$ is uniformly bounded in $\lambda \in \C^*$ with $|\lambda| \gtrsim 1.$ Every tangent current  $T_\infty$   to $T$ along $V$ satisfies
\begin{align} \label{eq_TifntysigmaE}
T_\infty= (\sigma_E)_* \widehat T_\infty
\end{align}
for some tangent current $\widehat T_\infty$ of $\widehat T$ along $\widehat V.$ There exists a closed positive current $\widehat S_\infty$ on $\widehat H_\infty$ such that 
\begin{align}\label{eq_TifntysigmaE22}
\widehat T_\infty= \widehat{\pi}^*_\infty \widehat S_\infty, \quad T_\infty= \pi^*_\infty S_\infty
\end{align}
for $S_\infty:= (\sigma_E)_* \widehat S_\infty.$   
\end{proposition}

Since $\pi_\infty$ is only a submersion from $E \backslash V $ to $H_\infty,$   in the second equality of (\ref{eq_TifntysigmaE22}),  the current $\pi^*_\infty S_\infty$ is \`a priori a closed positive current on $E \backslash V$ which can be extended to be a current on $E$ trivially through $V$ because it has locally bounded mass there. A direct consequence of (\ref{eq_TifntysigmaE22}) is that $T_\infty$ can be extended to be a closed positive current on $\overline E$ having  no mass on $V.$ We still have that the de Rham cohomology of every tangent current $T_\infty$ is the same  as in \cite{Dinh_Sibony_density}. But we don't need to use that fact in this paper. 

\proof   Although the desired assertions can be deduced more or less by using similar arguments from \cite{Dinh_Sibony_density}, we opt to give a complete proof for the readers' convenience. 

Let $\tau$ be as in the statement and $\widehat \tau: \widehat U \to \widehat E$ the lift of $\tau$ to $\widehat U= \sigma_E^{-1}(U)$ as above. By (\ref{eq_lifttau}) and the fact that $T, \widehat T$ have no mass on $V, \widehat V$ respectively, we have
\begin{align} \label{eq_taulift}
\big \langle (A_\lambda)_* \widehat \tau_* \widehat T, \sigma_E^* \Phi \big \rangle = \big\langle (A_\lambda)_* \tau_* T, \Phi \big\rangle
\end{align}
for every smooth form $\Phi$ with compact support in $E|_{V \cap U}.$   Recall that $\widehat \tau$ is almost-admissible.    By Lemma \ref{le_globaladmiss} and \textbf{(H)},  the mass of $(A_\lambda)_* \widehat \tau_* \widehat T$ is uniformly bounded on compact subsets of $\pi_{\widehat E}^{-1}(\widehat U \cap \widehat V).$ Using this and (\ref{eq_taulift}) implies the first desired assertion. 

Let $T_\infty$ be a tangent current with the defining sequence $(\lambda_n).$ By using a subsequence of $(\lambda_n)$ if necessary, we can assume also that $(A_{\lambda_n})_* \widehat \tau_* \widehat T$ converges to a tangent current $\widehat T_\infty$ of $\widehat T$ along $\widehat V.$  Substituting $\lambda=\lambda_n$ in (\ref{eq_taulift}) and letting $n \to \infty$ give
\begin{align} \label{eq_taulift2}
\big \langle \widehat T_\infty, \sigma_E^* \Phi \big \rangle = \big\langle T_\infty, \Phi \big\rangle.
\end{align}
Hence, the equality (\ref{eq_TifntysigmaE}) follows.  

We claim  that  $\widehat T_\infty$ have no mass on $\widehat V$. To see it, let $\big(\widehat U'_1 \times \widehat U''_1,(\widehat x',\widehat x_k)\big)$ be a  relatively compact local chart of $\widehat X$  with $\widehat x'=(\widehat x_1, \ldots, \widehat x_{k-1})$   so that $\widehat V$ is given by $\widehat x_k=0.$ 
Since the restriction of $\widehat T_\infty$ to $\widehat V$ is a closed positive current, it is the pushforward of a closed positive current on $\widehat V.$ It follows that the mass of $\widehat T$ on $\widehat V\cap (\widehat U'_1 \times \widehat U''_1)$ is 
\begin{align*}
& \lesssim \langle \widehat T_\infty, \mathbf{1}_{\widehat U'_1 \times \widehat U''_1}(\ddc \|\widehat x'\|)^{k-l} \rangle= \lim_{n \to \infty}\langle \widehat T, A_{\lambda_n}^*(\mathbf{1}_{\widehat U'_1 \times \widehat U''_1}(\ddc \|\widehat x'\|)^{k-l}) \rangle\\
&= \lim_{n \to \infty}\langle \widehat T, (\mathbf{1}_{\widehat U'_1 \times (\lambda_n^{-1} \widehat U''_1)}(\ddc \|\widehat x'\|)^{k-l}) \rangle=0
\end{align*}  
because $\widehat T$ has no mass on $\widehat V.$

We now prove (\ref{eq_TifntysigmaE22}).  To this end,  we first  check that $\widehat T_\infty$ is $\widehat V$-conic, \emph{i.e,} $(A_t)^* \widehat T_\infty = \widehat T_\infty$ for every $t \in \C^*$.  
Let the local coordinates be as above. Let $\Phi= \Phi_1(\widehat x') \wedge \Phi_2(\widehat x_k)$ be a smooth form with compact support on $\widehat U'_1 \times \C$.  Observe that 
$(A_t)^*\Phi_2$ and  $\Phi_2$ belongs to the same cohomology class with compact support in $\C$ because their integrals over $\C$ are equal.  

Since $H^2_c(\C, \C)$ is of dimension $1,$ there exists a smooth $1$-form $\Theta(x_k)$ with compact support on $\C$ for which $(A_t)^*\Phi_2- \Phi_2 -  d\Theta$ is a $1$-form in $x_k,$ \emph{i.e,} 
$$(A_t)^*\Phi_2- \Phi_2 - d\Theta= a(x_k) d x_k+ b(x_k) d\overline x_k.$$       
Using this, the closedness of $T$ and Cauchy-Schwarz inequality, we obtain that
\begin{align*}
\bigg|\big\langle (A_\lambda)_*T, \Phi_1 \wedge((A_t)^*\Phi_2- \Phi_2)  \big\rangle\bigg|  &=\bigg|\big\langle (A_\lambda)_*T, \Phi_1 \wedge ((A_t)^*\Phi_2- \Phi_2- d \Theta) \pm  d \Phi_1 \wedge \Theta  \big\rangle\bigg|  \\
&\lesssim \big\langle (A_\lambda)_* T, (\ddc \widehat x)^{k-l} \big\rangle \big\langle  (A_\lambda)_* T, \mathbf{1}_{\supp \Phi}(\ddc \widehat x')^{k-l}\big\rangle.
\end{align*}
The first term in the right-hand side of the last inequality is uniformly bounded in $\lambda$, whereas the second one converges to $0$ as $\lambda \to \infty$ because $(\ddc \widehat x')$ is invariant by $A_\lambda$ and $\supp (A_\lambda)^* \Phi$ converges to $V.$  Letting $\lambda \to \infty$ gives  $ (A_\lambda)_*T_\infty = T_\infty.$ 

Observe that if $\theta$ is a smooth closed positive form on $\widehat U_1:= \widehat U'_1 \times \widehat U''_1$ then $\widehat T_\infty \wedge \pi_{\widehat E}^*(\theta|_V)$ is a tangent current of $\widehat T \wedge \theta$ along $\widehat V$ on $\widehat U_1$ with the same defining sequence $(\lambda_n)$.  Choose a such $\theta$ of bidegree $(k-p,k-p).$ We obtain that $\mu:=\widehat T_\infty \wedge \pi_{\widehat E}^*(\theta|_{ \widehat V})$  is a nonnegative measure on $\pi_{\widehat E}^{-1}(\widehat U_1 \cap \widehat V)$ which has no mass on $\widehat V$ because  $\widehat T_\infty$ has no mass on $\widehat V.$  

 On the other hand,  since $\widehat T_\infty \wedge \pi_{\widehat E}^*(\theta|_{ \widehat V})$ is $\widehat V$-conic, for any smooth  positive cut-off function $\chi(\widehat x', \widehat x_k)$ supported on $\widehat U_1,$ we have 
$$\langle \mu, \chi \rangle= \langle \mu, A_t^* \chi \rangle$$
for every $t \in \C^*.$ Letting $t \to \infty$ in the last equality and observing that the limit supremum of the uniformly bounded functions $A_t^* \chi $ as $|t| \to \infty$ is  a function supported on $\pi_{\widehat E}(\supp \chi) \subset \widehat V \cap \widehat U_1,$ we obtain  
$$|\langle \mu, \chi \rangle| \lesssim \|\mathbf{1}_{\widehat V \cap \widehat U_1}\mu\|=0.$$
Thus, $\mu=0.$ Or in other words,  $\widehat T_\infty \wedge \pi_{\widehat E}^*(\theta|_{ \widehat V})=0$ for every smooth closed positive form $\theta$ of bidegree $(k-p,k-p).$ It follows that in the local coordinates $(\widehat x',t)$ of $\widehat E,$ the current $\widehat T_\infty$ must have the following form:
\begin{align}\label{eq_vietTinfyalpha}
\widehat T_\infty= \sum_{I, I'} \alpha_{I,I'}(\widehat x', t) d \widehat x'_I \wedge d \overline{\widehat x}'_{I'},
\end{align}
for some Radon measures $\alpha_{I,I'}$ on $\widehat E,$ where the sum is taken over $I,I'$ with $I, I' \subset \{1, \ldots, k-1\}$ and $I,I'$ are of cardinality $p.$  Since $\widehat T_\infty$ is closed, $\alpha_{I,I'}$ is independent of $t.$ As a result, we obtain a current $\widehat S_\infty$ on $\widehat U_1$ for which $\widehat T_\infty = \pi^*_{\widehat E} \widehat S_\infty$ on $\pi^{-1}_{\widehat E}(\widehat U_1).$ The last formula tells us that $\widehat S_\infty$ is independent of local coordinates and local charts. Hence, $\widehat S_\infty$ is a well-defined closed positive current on $\widehat V$ for which  $\widehat T_\infty = \pi^*_{\widehat E} \widehat S_\infty.$ 

Now notice that   $(\pi_{\widehat E})|_{\widehat H_\infty}$ is a biholomorphism between $\widehat H_\infty$  and $\widehat V.$ So we  can identify these two submanifolds via that biholomorphism. We then view $S_\infty$ as a current on $\widehat H_\infty.$ Observe now the fiber of $\pi_{\widehat E}$ at $\widehat x\in V$ only differs to the fiber of $\widehat \pi_\infty$ at $(\pi_{\widehat E})|_{\widehat H_\infty}^{-1}(\widehat x)$ at  two points. This implies that   $\widehat T_\infty = \widehat\pi^*_{\infty} \widehat S_\infty.$ This proves the first equality of (\ref{eq_TifntysigmaE22}). The second equality follows directly from the first one and the following formulae:
$$T_\infty=[(\sigma_E|_{\widehat E \backslash \widehat V})^{-1}]^* (\widehat T_\infty), \quad \widehat \pi_\infty \circ (\sigma_E|_{\widehat E \backslash \widehat V})^{-1}= (\sigma_E|_{\widehat E \backslash \widehat V})^{-1}\circ \pi_\infty.$$ 
This ends the proof.   
\endproof

Recall that $[\widehat V]$ is a current of bidegree $(1,1)$. Thus, $[\widehat V]$ can be represented as the sum $\ddc \widehat u+ \beta$ for some smooth form $\beta$ and some quasi-plurisubharmonic function $\widehat u$ on $\widehat X.$ A such $\widehat u$ is called a potential of $[\widehat V].$

\begin{proposition} \label{pro_uniquetangentcurent} Assume that $\widehat T \wedge [\widehat V]$ is well-defined in the classical sense, that means that  potentials of $[\widehat V]$ are locally integrable with respect to the trace measure of $\widehat T.$ Then the tangent current to $T$ along $V$ is unique and is given by  
$$T_\infty= \pi_\infty^* (\widehat T \wedge [\widehat V]),$$
where recall that we identified $\widehat V$ with $\P(E)$ and identified $\widehat T \wedge [\widehat V]$ with a current on $\widehat V.$  
\end{proposition}

\proof   By \cite{Viet_Lucas}, there is a  unique tangent current $R$ of $\widehat T \otimes [\widehat V]$ along the diagonal $\Delta_{\widehat X}$ of $\widehat X \times \widehat X$ and $R$ is equal to $\pi_{\Delta_{\widehat X}}^*(\widehat T \wedge [\widehat V])$ by identifying $\widehat X$ with $\Delta_{\widehat X},$ where  $\pi_{\Delta_{\widehat X}}$ is the projection from the normal bundle $E_{\Delta_{\widehat X}}$ onto $\Delta_{\widehat X}.$  Hence, by Lemma \ref{le_deltatangecurr}, there exists uniquely a tangent current of $\widehat T$ along $\widehat V$ which is given by $\pi_{\widehat E}^*(\widehat T \wedge [\widehat V]).$ This combined with (\ref{eq_TifntysigmaE22}) yields that
$$S_\infty=\widehat T \wedge [\widehat V] .$$
The desired equality then follows. 
 The proof is finished.
\endproof



Let $\omega$ be a positive definite Hermitian form on $X.$ Let $U$ be a relatively  compact subset of $X.$  If $\codim V \ge 2,$ let $\widehat\omega_h$ be a Chern form of $\mathcal{O}(-\widehat V)$ whose restriction to each fiber of  $\widehat{V} \approx \P(E)$ is strictly positive, otherwise we simply put $\widehat \omega_h:= 0.$    By scaling $\omega$ if necessary, we can assume that  $\widehat \omega:=  \sigma^* \omega+ \widehat \omega_h>0$ on $\widehat U:=\sigma^{-1}(U).$   The following result generalizes \cite[Th. 4.6]{Dinh_Sibony_density}. 

\begin{theorem} \label{th_Hhold} Let $X$ be a complex manifold.  Let $T$ be a closed positive current of bidimension $(q,q)$ on $X$ for $q \ge 0$  and $V$ a smooth submanifold of $X.$  Assume that 

$(i)$ $\supp T \cap V$ is compact and $T$ has no mass on $V$,

$(ii)$ there exists a Hermitian form $\omega$ on $X$ for which $\ddc \omega^j=0$ on $V$ for $1\le j \le q-1.$
\\
Then Assumption \textbf{(H)} holds for $X,V,T$. Moreover, given any compact $K$ and relatively compact open subset $U$ in $X$ such that  $K\subset U$ and $\supp T \cap V \subset K$,    there exists a constant $c$ independent of $T$ for which 
\begin{align}\label{ine_daudauth27Tnga}
\|\widehat T\|_{\widehat K} \le c \|T\|_{U},\quad  \|T_\infty\| \le c \|T\|_U,
\end{align}
for every tangent current $T_\infty$ of $T$ along $V,$ where $\widehat K:= \sigma^{-1}(K).$ 
\end{theorem}

Here for every   current  $S$ of order $0$ on $X$ and $U \subset X,$   $\|S\|_U$ denotes the mass of $S$ on $U.$

\proof  We use ideas from \cite{Dinh_Sibony_density}.  If $q=0$, then $T$ is a measure having no mass on $V$ and $T_\infty=0.$ The desired assertions obviously hold.    Consider now $q\ge 1.$ We first show that the mass of $\widehat T$ is locally bounded near $\widehat V.$    Let $W, W_T$ be  open neighborhoods of $V, \supp T$ respectively  such that $W_T\cap W$ is relatively compact in $X$.    Fix $\omega$ as in Assumption $(ii)$ and $\widehat \omega_h, \widehat \omega$ as above. We can assume also that  $K\subset W_T\cap W\Subset U$.  

Let $\varphi$ be a quasi-p.s.h. function $\varphi$ on $X$ such that $\supp \varphi \subset W$ and  
\begin{align}\label{eq_bieuthucomegah}
\sigma_* \widehat \omega_h= \ddc \varphi + \eta
\end{align}
 for some smooth closed form $\eta$ on $X.$ By multiplying $\widehat \omega_h$ by a strictly positive constant, we have  $$\sigma^* \sigma_* \widehat \omega_h= \widehat \omega_h+  [\widehat V]$$
 if $\codim V \ge 2$. Indeed one only needs to check it locally. Hence, we can reduce this question to the K\"ahler case where the desired identity is already known. Thus  we have
\begin{align} \label{ine_singvarphi}
|\varphi(x) - \log\dist(x, V)|  \lesssim 1
\end{align}
on compact subsets of $U$ provided that $\codim V \ge 2.$ 

Put $\widehat W: = \sigma^{-1}(W)$ and $\widehat W_T: = \sigma^{-1}(W_T).$ Clearly $\widehat W_T \cap \widehat W \Subset \widehat X.$ If $V$ is a hypersurface, the first inequality of (\ref{ine_daudauth27Tnga}) is clear because $\widehat T= T.$ Consider $\codim V \ge 2.$   Since $\widehat T$ has no mass on $\widehat V,$ using the fact that $\sigma_* \widehat{\omega}_h$ is smooth outside $V$ and $\sigma_* \sigma^* \omega=\omega,$ we have 
\begin{align} \label{eq_xuongXV}
\int_{\widehat K} \widehat T \wedge \widehat{\omega}^{q}&= \int_{K \backslash V} T \wedge \sigma_*(\widehat{\omega}^q)= \int_{K \backslash V} T \wedge (\ddc \varphi+ \eta+ \omega)^q \\
\nonumber
&\le \int_{U \backslash V} T \wedge (\ddc \varphi+ c \omega)^q,
\end{align}
for some  positive constant $c$ independent of $T$ with $\ddc\varphi+ c \omega \ge 0.$ 

For a positive constant $M,$ put $\varphi_M:= \max\{\varphi, -M\}.$ Note that $\ddc \varphi_M +c \omega \ge 0.$  Since  the positive current $T \wedge (\ddc \varphi_M+ c \omega)^q$ converges to $T \wedge (\ddc \varphi+ c \omega)^q$ on $X \backslash V$ as $M \to \infty,$ using (\ref{eq_xuongXV}), we get
\begin{align} \label{eq_xuong2}
\|\widehat T\|_{\widehat K} \lesssim  \liminf_{M \to \infty} \int_{U \backslash V} T \wedge (\ddc \varphi_M+ c \omega)^q \le \liminf_{M \to \infty} A_M,
\end{align}
where $$A_{M}:= \int_{U} T \wedge \big(\ddc \varphi_M+c \omega\big)^q.$$   Using Assumption $(ii),$  we observe that  $A_M$ can be written as  a linear combination of
$$A_{M,j}:=\int_{U} T \wedge (\ddc \varphi_M)^j \wedge \omega_j \quad (0 \le j \le q)$$
with coefficients of absolute values bounded a constant independent of $T,$ where $\omega_j$ is a  smooth $(q-j,q-j)$-form  depending only on $\omega$  such that   $\ddc \omega_j=0$ on $V.$  So we only need to bound $A_{M,j}.$ We will prove that 
\begin{align}\label{eq_AphayMj}
|A_{M,j}| \lesssim \|T\|_U
\end{align}
for $0 \le j \le q$  by induction on $j.$   When $j=0,$   (\ref{eq_AphayMj}) is also clear.   Assume that (\ref{eq_AphayMj}) holds for every $1, \ldots, j-1.$ This in particular implies 
\begin{align} \label{ine_iniduchyp}
\|T \wedge \big(\ddc \varphi_M+c \omega\big)^{j-1}\| \lesssim \|T\|_U.
\end{align}
Since $T \wedge (\ddc \varphi_M)^{j-1}$ can be written as a linear combination of $T \wedge \big(\ddc \varphi_M+c \omega\big)^{s}$ for $0 \le s \le l-1,$ we deduce from (\ref{ine_iniduchyp}) that 
\begin{align} \label{ine_iniduchyp2}
\|T \wedge (\ddc \varphi_M)^{j-1}\| \lesssim \|T\|_U.
\end{align}
By (\ref{ine_singvarphi}), for $M$ big enough we have  $\supp T \cap \supp \varphi_M \subset W_T \cap W$ which is compact in $X.$  Thus using Stokes' theorem, one obtains
\begin{align*}
A_{M,j} &= \int_U \varphi_M\ddc (T \wedge \omega^j) \wedge (\ddc \varphi_M)^{j-1} \\
&= \int_U  T \wedge (\ddc \varphi_M)^{j-1} \wedge \varphi_M \ddc \omega_j \le  \|  \varphi_M \ddc \omega_j  \|_{\cali{C}^0}  \| T \wedge (\ddc \varphi_M)^{j-1}\|_U\\
& \lesssim \|  \varphi_M \ddc \omega_j  \|_{\cali{C}^0} \|T\|_U
\end{align*}
 because of (\ref{ine_iniduchyp2}).  By (\ref{ine_singvarphi}) and the fact that $\ddc\omega_j=0$ on $V$,  the $\cali{C}^0$-norm of  the form $\varphi_M \ddc \omega_j$ is  bounded independently of $M$.  Thus we get (\ref{eq_AphayMj}) for $j.$ It follows that  $A_M \lesssim \|T\|$ for $M$ big enough. This combined with (\ref{eq_xuong2}) gives $\|\widehat T\|_{\widehat K} \lesssim \|T\|_U.$

Now it remains to show that $(A_\lambda)_* (\widehat \tau_j)_* \widehat T$ is of uniformly mass  on compact subsets of $\pi_{\widehat E}^{-1}(\widehat U_j)$ for some suitable holomorphic admissible maps $\widehat \tau_j: \widehat U_j \to \widehat E$ with $\widehat V \subset \cup_j \widehat U_j.$ Notice that $T_\infty$ is supported in $\pi^{-1}(\supp T \cap V) \subset \widehat K.$

Let  $(\widehat U_j)_j$ be  a family of local charts biholomorphic to $\D^k$ covering $\widehat V \cap W_T$ and $\widehat x=(\widehat x_1, \ldots, \widehat x_k)$ local coordinates on $\widehat U_j$ such that $\widehat V\cap \widehat U_j$ is given by $\widehat x_1=0.$ Identify $\widehat E$ with $\widehat V \times \C.$ Let $\widehat \tau_j: \widehat U_j \to \widehat V \times \C$ be the identity map. 

Let $\rho:= d \widehat x_1$ or $d\overline{\widehat x}_1.$ Since $T_\infty$ is the pullback of a current on $H_\infty$ via $\pi_\infty,$ in order to bound the mass of $T_\infty,$ it suffices to bound  the mass of $|\lambda|(\widehat T \wedge \rho)$ on $Z_\lambda:= \{|\widehat x_1| \le |\lambda|\}$ and   and  the mass of  $|\lambda|^2(\widehat T \wedge \ddc |\widehat x_1|^2)$   on $Z_\lambda.$   By the Cauchy-Schwarz inequality,  
$$\|\widehat T \wedge \rho\|_{Z_\lambda} \lesssim  \|\widehat T\|^{1/2}_{Z_\lambda} \,\|\widehat T \wedge \ddc |\widehat x_1|^2 \|^{1/2}_{Z_\lambda}.$$
Thus, it is enough to estimate $\|\widehat T \wedge \ddc |\widehat x_1|^2\|_{Z_\lambda}.$ This is already done in \cite{Dinh_Sibony_density}. We reproduce arguments here for the readers' convenience. 

Put $\widehat W:= \sigma^{-1}(W).$  Let $u$ be a quasi-p.s.h. function on $\widehat X$ such that $u$ vanishes outside $\widehat W$ and $u$ is a potential of  $[\widehat V]$ on  a small enough open neighborhood  $\widehat W_1 \Subset \widehat W$ of $\widehat V$ \emph{i.e,} $[\widehat V]= \ddc u+ \beta,$ for some smooth form $\beta$ on $\widehat W.$ Note that $\widehat W \cap \supp T$ is relatively compact in $\widehat X.$  We have 
$$\big|u(\widehat x) - \log|\widehat x_1| \big| \le A$$
on $\widehat U_j$ for every $j$ and some constant $A$ independent of $j.$  

Put $s:= \log |\lambda|.$    We only need to consider $|\lambda|$ big, say, $|\lambda|\ge e^{3A}.$ Thus $u \le -\log |\lambda| + 2 A$ on $Z_\lambda.$  Let $\chi_\lambda$ be  a convex increasing function bounded from below on $\R$ such that  $\chi_\lambda(t)=t$ for  $t \ge -\log |\lambda|+3A$ and $0 \le \chi' \le 1$ and 
$$\chi_\lambda''(t)= e^{2t+2\log |\lambda|-5A}$$
 for $t \le -\log |\lambda| +2A.$ Put $\phi_\lambda:= \chi_\lambda \circ u$ which is bounded and  supported on $ \widehat W.$ 
Hence we get
\begin{align} \label{incluu}
\supp \phi_\lambda \cap \supp \widehat T \Subset \widehat X.
\end{align}
We also have 
\begin{align} \label{ine_philambdasingu}
|\phi_\lambda(\widehat x)| \lesssim \big|\log|\widehat x_1| \big|.
\end{align}
Direct computations (see \cite[Le. 2.11]{Dinh_Sibony_density}) give
$$\ddc \|\widehat x_1\|^2 \le c |\lambda|^{-2}(\ddc \phi_\lambda+ \widehat\omega)$$
on $Z_\lambda$ for some constant $c$ independent of $\lambda.$ This  implies that
\begin{align*} 
\|\widehat T \wedge \ddc |\widehat x_1|^2 \|_{Z_\lambda}= \int_{Z_\lambda} T \wedge \ddc |\widehat x_1|^2 \wedge \widehat \omega^{q-1} \lesssim |\lambda|^{-2} \int_{\widehat W_T \cap \widehat W} \widehat T \wedge (\ddc \phi_\lambda+ \widehat \omega) \wedge \widehat \omega^{q-1}. 
\end{align*}
Observe that $\ddc (\widehat \omega^{q-1})=0$ on $\widehat V.$  This together   with (\ref{ine_philambdasingu}) and  (\ref{incluu})  allows us to argue as before to obtain that 
$$ \|\widehat T \wedge \ddc |\widehat x_1|^2 \|_{Z_\lambda} \lesssim  |\lambda|^{-2} \|\widehat T\|_{\widehat W_T \cap \widehat W} \lesssim   |\lambda|^{-2} \|T\|.$$
The proof is finished.
\endproof

We now study the continuity of tangent currents as $T$ converges to another current. Such a property is crucial in applications.   When $X$ is K\"ahler, we give a sufficient condition ensuring the \emph{continuity} of the total tangent classes, see Proposition  \ref{pro_continui-totaltangenclass} below.  On the other hand, for a general non-K\"ahler compact manifold $X$, the de Rham cohomology class of  a positive closed current can be vanished. 
So  the use cohomology classes of closed positive currents in a general complex manifolds  is not as  efficient as usual.  Proposition \ref{pro_semicontinuity2} below   serves as a substitute for the semi-continuity theorem in  \cite[Th. 4.11]{Dinh_Sibony_density}.  
 
\begin{lemma} \label{le_limV} Let $Y$ be a complex manifold and $Z$ a complex hypersurface of $Y$. Let $\xi$ be a smooth closed $(1,1)$-form on $Y$ such that $[Z]= \ddc u+ \xi$ for some   quasi-p.s.h. function $u$ which is smooth outside $Z$ and has a log singularity near $Z.$   Let $R_n$ be a sequence of closed positive currents of bidimension $(q,q)$ on $Y$ converging to a current $R_\infty.$ Assume that 

$(i)$ there exists a compact $K$ of $Y$ for which  $\supp R_n \cap Z \subset K$ for every $n$,

$(ii)$  $R_n \wedge [Z]$ and $(\mathbf{1}_{Y \backslash Z} R_\infty) \wedge [Z]$ are classically well-defined for every $n,$ where $\mathbf{1}_{Y \backslash Z}$ is the characteristic function of $Y \backslash Z.$ \\
Then for every smooth   $2(q-1)$-form $\Omega$  on $Y$ with  $\ddc \Omega=0$ on $Z,$ we have
\begin{align} \label{limit_RnloaiZ}
\lim_{n\to \infty} \langle R_n \wedge [Z], \Omega \rangle = \big\langle (\mathbf{1}_{Y \backslash Z} R_\infty) \wedge [Z], \Omega \big\rangle+ \langle \mathbf{1}_Z R_\infty, \xi \wedge \Omega\rangle,
\end{align}
where we extended the closed positive current $\mathbf{1}_{Y \backslash Z} R_\infty$ trivially through $Z$ to obtain a current on $Y$ and $\mathbf{1}_Z R_\infty$ is viewed as a current on $Z.$ 
\end{lemma}

\proof By $(i)$  there is a cut-off function $\chi$ compactly supported on $Y$ such that $0 \le \chi \le 1$ and $\chi \equiv 1$ on an open neighborhood $W$ of $Z$ with $W \cap \supp R_n \Subset K'$ for every $n$ and some fixed compact $K'$ independent of $n.$   Since the support of $R_n \wedge [Z]$ is contained in $W,$  we get 
\begin{align*}
\langle R_n \wedge [Z], \Omega \rangle &=\langle R_n \wedge [Z], \chi \Omega \rangle=\langle R_n \wedge (\ddc u+ \xi), \chi \Omega \rangle\\
&= \langle R_n \wedge  \xi, \chi \Omega \rangle+\int_Y R_n \wedge (u \chi \ddc \Omega)+ \\
&+ \int_Y u \, R_n \wedge \big(\ddc\chi\wedge  \Omega- \dc \chi \wedge d \Omega+ d \chi \wedge \dc \Omega\big).
\end{align*}
Denote by $I_1, I_2, I_3$ respectively the first, second and third terms in the right-hand side of the last formula. Clearly,
$$ I_1 \to \langle R_\infty \wedge  \xi, \chi \Omega \rangle=\langle (\mathbf{1}_{Y \backslash Z}R_\infty) \wedge  \xi, \chi \Omega \rangle+\langle \mathbf{1}_Z R_\infty \wedge  \xi, \Omega \rangle$$
and
$$I_2 \to  \int_Y u \, (\mathbf{1}_{Y \backslash Z} R_\infty) \wedge \big(\ddc\chi\wedge  \Omega- \dc \chi \wedge d \Omega+ d \chi \wedge \dc \Omega\big)$$
because $u$ is smooth outside $Z$ and $\d \chi$, $\dc \chi$ vanish near $Z.$ On the other hand, since $\ddc \Omega=0$ on $Z,$ the form $u \chi \ddc \Omega$ is continuous on $Y$ and equal to $0$ on $Z.$  This implies that 
$$I_2 \to \int_Y R_\infty \wedge (u \chi \ddc \Omega)= \int_{Y} (\mathbf{1}_{Y \backslash Z} R_\infty) \wedge (u\chi \ddc \Omega).$$
The desired limit (\ref{limit_RnloaiZ}) then follows.  
The proof is finished.
\endproof

\begin{definition}(\cite{Dinh_Sibony_density})  For a positive current $R$ in $\P(E),$ the \emph{h-dimension} of $R$ is the biggest integer $s$ for which $R\wedge \pi^* \omega^s_V \not = 0,$ where $\omega_V$ is a Hermitian metric on $V.$ 
\end{definition}

Observe that by a bidegree reason, if $(p,p)$ denotes the bidegree of $R,$ then the h-dimension of $R$ is at least $(l-p),$ where $l$ is the dimension of $V.$    The following is a version of the semi-continuity theorem  \cite[Th. 4.11]{Dinh_Sibony_density}.

\begin{proposition} \label{pro_semicontinuity2}  Let $X$ be a  complex manifold.  Let $(T_n)_{n\in \N}$ be a sequence of closed positive currents of bidimension $(q,q)$ converging to  a current $T.$ Assume that

$(i)$   $T_n$ has no mass on $V$ and $\supp T_n \cap V \subset K$ for some compact $K\subset U$ independent of $n,$

$(ii)$  the products $\widehat T_n \wedge [\widehat V],$ $\widehat T_\infty\wedge [\widehat V]$ are well-defined in the classical sense for every $n,$

$(iii)$ $\ddc \omega^j=0$ on $V$ for $1 \le j \le q-1.$\\
Then for 
$$\widehat T_{n,j}:=\widehat T_n \wedge [\widehat V] \wedge \widehat \omega^{q-j-1}, \quad  \widehat T_{\infty,j}:=\widehat T_\infty \wedge [\widehat V] \wedge \widehat \omega^{q-j-1}$$
then the following three properties hold:

$(i)$ \begin{align} \label{conver_TV}
\lim_{n \to \infty} \sigma_* \widehat T_{n,j} =0  
\end{align}
for $j> j_\infty$, where $j_\infty$ is the h-dimension of $\widehat T_\infty \wedge [\widehat V]$ (recall $\widehat  V \approx \P(E)$),

$(ii)$ \begin{align} \label{conver_TVsualai}
\limsup_{n \to \infty} \langle \sigma_*\widehat T_{n,j_\infty}, \omega^{j_\infty} \rangle \le  \langle \sigma_*\widehat T_{\infty,j_\infty}, \omega^{j_\infty} \rangle.
\end{align}

$(iii)$ if moreover $q+l\ge k+j_\infty,$ we have 
\begin{align} \label{conver_TVsualaicomplementary}
\lim_{n \to \infty} \langle \sigma_*\widehat T_{n,j}, \omega^{j} \rangle =  \langle \sigma_*\widehat T_{\infty,j},\omega^{j} \rangle
\end{align}
for every $0 \le j \le q-1.$
\end{proposition}

Note that if $q\ge 3,$ Condition $(iii)$ is equivalent to that $\ddc \omega= \partial \omega \wedge \overline\partial \omega=0$   on $V.$  If $T_n$ are currents of integration along analytic sets, then the assumption of Proposition \ref{pro_semicontinuity2} on $T_n$ is automatically satisfied. This is the case in our application to the problem of estimating the number of isolated periodic points of meromorphic self-maps later. 

\proof By extracting a subsequence, we can assume that $\widehat T_n \to \widehat T'$. We have $\mathbf{1}_{\widehat X \backslash \widehat V}\widehat T'= \widehat T.$ Put $\widehat S:= \mathbf{1}_{\widehat V} \widehat T'.$   Denote by $j_n$ the h-dimension of $\widehat T_n \wedge [V]$ and $j_{\widehat S}$ the h-dimension of $\widehat S.$ By extracting a subsequence, without loss of generality, we can assume that $j_n$ are all equal and denote by $j^*$ this number.

Notice that $[\widehat V]= \ddc u - \widehat \omega_h$ for some quasi-p.s.h. $u$ by the choice of $\widehat \omega_h.$   
By $(iii),$   $\ddc \widehat \omega^j=0$ on $\widehat V$ for $1 \le j \le q-1.$ Hence if $\theta:=  \omega^{j}$ or $\theta$ is a closed smooth form of $X,$ then $\Omega:= \widehat \omega^{q-j-1} \wedge \sigma^* \theta$ is $\ddc$-closed.  This allows us to apply Lemma \ref{le_limV} to $Y:=\widehat X,$ $Z:= \widehat V, \xi:= - \widehat \omega_h,  R_n:=\widehat T_n$ and $\Omega.$ We then obtain that 
\begin{multline}\label{eqlimit-TnTV}
\lim_{n\to \infty}\int_{\widehat X} \widehat T_n \wedge [\widehat V] \wedge \widehat \omega^{q-j-1} \wedge \sigma^* \omega^{j}= \int_{\widehat X }\widehat T_\infty \wedge [\widehat V] \wedge \widehat \omega^{q-j-1} \wedge \sigma^* \omega^{j}+\\
- \int_{\widehat V} \widehat S \wedge  \widehat \omega_h \wedge  \widehat \omega^{q-j-1} \wedge \sigma^* \omega^{j}.
\end{multline}
Denote by $I_{n}(j), I_\infty(j), I_{\widehat S}(j)$ respectively the first, second and third integrals in the above limit (from left to right).  We have $I_n, I_\infty\ge 0$ and  $\lim_{n\to \infty} I_n= I_\infty - I_{\widehat S}.$
It follows that if $j_{\widehat S} >j_\infty,$ then  we get $I_n(j_{\widehat S}) \ge 0,  I_\infty(j_{\widehat S})=0$ and $I_{\widehat S}(j_{\widehat S})>0.$ This is a contradiction. Thus 
\begin{align} \label{ine-comparehdim}
j_{\widehat S} \le   j_\infty.
\end{align}
Thus  (\ref{conver_TV})  follows. The second desired limit (\ref{conver_TVsualai}) is deduced from (\ref{eqlimit-TnTV}) and the fact that
$$\int_{\widehat V} \widehat S \wedge  \widehat \omega_h \wedge  \widehat \omega^{q-j-1} \wedge \sigma^* \omega^{j}=\int_{\widehat V} \widehat S \wedge  \widehat \omega^{q-j} \wedge \sigma^* \omega^{j}\ge 0 $$
if $j\ge j_{\widehat S}.$ 

Consider now the case where $q+l \ge k+j_\infty.$  We will prove that $\widehat S=0.$   We already know from (\ref{ine-comparehdim})  that its h-dimension $j_{\widehat S}$ is at most $j_\infty.$ Thus 
\begin{align}\label{eqsuytubdthdimS}
\widehat S \wedge \sigma^*\omega^j=0
\end{align}
 for every $j>j_\infty.$  Let $\widehat W_s\approx  W_s \times \D^{k-l-1}$ be a local chart of $\widehat V$  which trivialize the projection $\sigma|_{\widehat V},$ where $W_s$ is a local chart of $V$ and $\D$ is the unit disk in $\C$.  Let $\omega_0$ be the standard K\"ahler form on $\C^{l-1}.$  Define 
 $$\omega':=\omega+\omega_0$$
which is a Hermtian metric on $\widehat W$ satisfying  
$$\omega'^{k-l-1+j}=  \omega^j \wedge \beta_j,$$
 for some smooth form $\beta_j$ because $\omega_0^{k-q}=0.$   Using $q\ge k-l+j_\infty$ and the fact that the bidimension of $\widehat S$ is $(q,q),$ we obtain that the mass of $\widehat S$ on compact subset of $\widehat W$ is bounded by a constant times  
$$\langle \widehat S,  \omega'^q \rangle=\langle \widehat S,  \omega'^{k-l-1+j_\infty+1} \wedge \omega'^{q-k+l-j_\infty} \rangle=\langle \widehat S\wedge  \sigma^*\omega^{j_\infty+1}, \beta_j \wedge \omega'^{q-k+l-j_\infty} \rangle=0$$
by (\ref{eqsuytubdthdimS}).  Thus $\widehat S=0$ on $\widehat W.$ It follows that $\widehat S=0$. Combining this with (\ref{eqlimit-TnTV})  gives (\ref{conver_TVsualaicomplementary}).  The proof is finished.
\endproof

We in fact obtain the following continuity property for total tangent classes in the K\"ahler case. 

\begin{proposition}\label{pro_continui-totaltangenclass} Let $X$ be a  complex K\"ahler  manifold.  Let $(T_n)_{n\in \N}$ be a sequence of closed positive currents of bidimension $(q,q)$ converging to  a current $T.$ Assume that

$(i)$   $T_n$ has no mass on $V$ and $\supp T_n \cap V \subset K$ for some compact $K\subset U$ independent of $n,$

$(ii)$  $q+l \ge k+j_\infty$, where $j_\infty$ is the h-dimension of the total tangent class $\kappa^V(T_\infty)$ of $T_\infty$ along $V$.

Then if $\kappa^V(T_n)$  denotes the total tangent class of $T_n$ along $V,$  then  we have  
\begin{align} \label{conver_TV2}
\lim_{n \to \infty} \kappa^V(T_n)  =   \kappa^V(T_\infty).
\end{align}
\end{proposition}

Note that since $j_\infty$ is at least $(l-p)$ where $(p,p)$ is the bidegree of $T_\infty,$  the above condition $(ii)$ is equivalent to $j_\infty= q+l-k.$  If $T_\infty$ is the current of integration along an analytic set, then the equality $j_\infty= q+l-k$ means $T_\infty$ intersects $V$ properly by Theorem \ref{th_V1alongV} and Lemma \ref{le_tontaiWnga}  in the next section. 
 
\proof   Recall that the total tangent class of $T_\infty$ along $V$ is a cohomology class  in $H_\infty \approx \P(E)$  and is given by $\pi_\infty^*(\{\widehat T_\infty\}|_{\widehat V})$ where we identified $\widehat V$ with $\P(E)$  and $\pi_\infty: \overline E\backslash V \to H_\infty$  is the natural projection.  We still have a notion of h-dimension for such a class, see \cite[Def. 3.7]{Dinh_Sibony_density} and the notions of h-dimensions of positive closed currents on $\P(E)$ and their cohomology classes are the same.    

Consider a tubular neighborhood $\widehat W$ of $\widehat V,$ that means $\widehat W$ is diffeomorphism to an open neighborhood of  $\widehat V$ in the normal bundle of $\widehat V.$ Using the projection of the normal bundle of $\widehat V,$ we get a projection $\widehat \pi: \widehat W \to \widehat V.$ Observe that $W:=\sigma(\widehat W)$ is an open neighborhood of $V$ in $X$ and $\widehat \pi$ induces a projection $\pi: W \to V.$ Since out problem is of local nature near $V,$ from now on we restrict to working on $W, \widehat W.$  As above we can assume $\lim_{n \to\infty} \widehat T_n= \widehat T_\infty+ \widehat S,$ where $\widehat S$ is a current on $\widehat V.$

Recall that for every cohomology class $\alpha$ in $W,$ we have 
$$\alpha|_V= \pi_*(\alpha \wedge \{V\}).$$
We can see it by using Poincar\'e's duality on $V$ and  de Rham's regularisation theorem to get a closed form representing $\{V\}$ with support closed to $V$.   Using the last formula and 
 (\ref{eqlimit-TnTV}) with $\omega^j$ replaced  by $\pi^* \theta$ for  an arbitrary smooth closed form $\theta$ in $V,$ one gets
\begin{align}\label{lim-kappaTnSmu}
\lim_{n\to \infty} \kappa^V(T_n) \wedge  (\widehat \omega|_{\widehat V})^{q-j-1} = \kappa^V(T_\infty) \wedge  (\widehat \omega|_{\widehat V})^{q-j-1}- \{\widehat S\}\wedge \widehat \omega_h \wedge  (\widehat \omega|_{\widehat V})^{q-j-1}
\end{align}
for every $0 \le j \le q-1.$  Now using $q+l \ge k+ j_\infty$ and arguing exactly as in the proof of Proposition \ref{pro_semicontinuity2} give $\widehat S=0.$ Combining this with the fact that (\ref{lim-kappaTnSmu}) still holds if we replace $\widehat \omega$ by $\widehat \pi^* \widehat \theta,$ for  any closed smooth form $\widehat \theta$ in $\widehat V,$ we obtain the desired limit. This finishes the proof.
\endproof

 \section{Intersection of analytic sets} \label{sec_analy} 

In this section, we study the intersection of analytic sets by using the  theory of tangent currents developed in the last section.  Our  first main result in this section  is  Theorem \ref{th_V1alongV} below saying that  the tangent current of an analytic subset  along a smooth submaniold on \emph{every complex manifold} always exists and is unique.   Let $X$ be a complex manifold. We emphasize that there is neither compactness assumption nor K\"ahler condition on $X.$ Let $V, E, \sigma, \widehat X, \widehat V, \pi_\infty$ be as in the last section. 
  Let $V_1 \not \subset V$ be an analytic subset of $X$ and $\widehat V_1$ the strict transform of $V_1$ in $\widehat X.$

 
 \begin{theorem} \label{th_V1alongV}  The tangent current of $[V_1]$ along $V$ is unique and is given by the pull-back of $[\widehat V_1] \wedge [\widehat V]$  by $\pi_\infty$. As a consequence, for  analytic subsets $V_1, V_2$ of $X,$ the density current associated to $[V_1], [V_2]$ is unique. 
\end{theorem}

\proof Clearly, $\widehat V_1$ is not a subset of $\widehat V.$ Thus, $\widehat V_1$ intersects $\widehat V$ properly because $\widehat V$ is a hypersurface. We deduce that the  wedge product $[\widehat V_1] \wedge [\widehat V]$ is well-defined in the classical sense, see \cite{Fornaess_Sibony,Demailly_ag}. The desired assertion then follows immediately from Proposition \ref{pro_uniquetangentcurent}. The proof is finished. 
\endproof

Put $l_1:= \dim V_1.$ Denote by $\mathcal{W}$ the set of irreducible components of  $\widehat V_1 \cap \widehat V$. These components are of dimension $(l_1-1).$  Write 
$$[\widehat V_1] \wedge [\widehat V]= \sum_{\widehat W \in \mathcal{W}} \alpha_{\widehat W} \widehat W,$$
for some nonnegative numbers $\alpha_{\widehat W}.$ Recall that  given hypersurfaces $D_1, \ldots, D_m$ and an analytic subset $D$, if  $D_1, \ldots,D_m, D$ intersects properly, then the intersection $D_1 \wedge \cdots \wedge D_m \wedge D$ in the sense of the pluripotential theory is the same as that defined in the classical sense of the theory of the intersection of analytic sets. The reason is that the both definitions enjoy the same continuity property, see \cite[p. 212]{chirka}. Thus  $\alpha_{\widehat W}$ is equal to the usual multiplicity along $\widehat W$ of the proper intersection $\widehat V_1 \cap \widehat V.$ In particular, $\alpha_{\widehat W}$ is a strictly positive  integer for every  $\widehat W \in \mathcal{W}.$

We now study the relation between irreducible components $\widehat{W}$ and their images by $\sigma.$ This will be crucial for our applications.  A typical situation we should keep in mind is the case where $V_1,V$ are of complementary dimensions and \emph{they do not intersect properly.} However, in what follows, we only assume the complementary dimension condition when necessary. 

Denote by $\mathcal{W}_1$ the set of all irreducible components of $V_1 \cap V.$ We first begin with the following lemma. 

\begin{lemma} \label{le_tontaiWnga} Let $W_1\in \mathcal{W}_1.$ Then  there exists  $\widehat W_1 \in \mathcal{W}$ such that 
 $W_1= \sigma(\widehat W_1).$ 
\end{lemma}
 
\proof For every $\widehat W \in \mathcal{W},$  $\sigma(\widehat W)$ is an irreducible analytic subset of $V_1\cap V$ because otherwise if $A$ is the singular part of $\sigma(\widehat W)$ then $\sigma(\widehat W)\backslash A$ has at least two connected components, hence,  so does 
$$\widehat W \backslash \sigma^{-1}(A)= (\sigma|_{\widehat W})^{-1}( \sigma(\widehat W)\backslash A);$$
 this is a contradiction by the irreducibility of $\widehat W.$   

 Since $\sigma(\widehat V_1)=V_1$ and $\sigma$ is bijective outside $\widehat V,$ we have   $\sigma(\widehat V_1 \cap \widehat V)= V_1 \cap V.$  This combined with the irreducibility of $\sigma(\widehat W)$ implies  that there exists a finite number of elements of $\mathcal{W}$ the union of  whose images by $\sigma$ is $W_1.$ Since these images are analytic subsets of $W_1,$ we deduce that at least one of them is  $W_1.$ The existence of $\widehat W_1$ then follows.  
\endproof

Notice that the above result doesn't rule out the possibility that there exists $\widehat W \in \mathcal{W}$ such that $\sigma(\widehat W)$ is a proper analytic subset of an irreducible component of $V_1 \cap V.$ This is the reason for the following definition. \emph{An exotic (intersection) component} $W$ of $V_1\cap V$ is the image $\sigma(\widehat W)$ of some $\widehat W \in \mathcal{W}$ such that $\sigma(\widehat W)\not \in\mathcal{W}_1.$ Denote by $\tilde{\mathcal{W}}_1$ the union of $\mathcal{W}_1$ with the set of all exotic components of $V_1 \cap V.$ For $W\in\tilde{\mathcal{W}}_1,$ we define \emph{its multiplicity} as 
$$\nu_W:=  \sum_{\widehat W: \,\sigma(\widehat W)=W} \alpha_{\widehat W}.$$
In this paper, we will pay special attention to components of dimension zero. An \emph{exotic (intersection) point $x$} of $V_1\cap V$ is an exotic 0-dimensional component of $V_1 \cap V.$ 

Let $P$ be  \emph{the number of isolated intersection points} counted with multiplicity of $V_1 \cap V$ and $\tilde{P}$ \emph{the sum of $P$ with the number of exotic intersection points} counted with multiplicity in $V_1 \cap V.$ One of our goals is  to estimate $\tilde{P}.$  We will need the following auxiliary result.

\begin{lemma} \label{le-submer-estimat} Let $Y,Z$ be two K\"ahler manifolds and $g: Y \to Z$ a proper holomorphic submersion. Let $K$ be a compact in $Z.$ Let $Q$ be an irreducible analytic subset of $Y.$ Then  generic fibers of $g|_Q$ have the same volume denoted by $\beta_Q$ and  there exists a constant $M$ independent of $Q$ such that we have 
\begin{align}\label{ine-estiQgQ}
M^{-1} \beta_Q \vol\big(g(Q)\cap K\big)\le \vol(Q\cap g^{-1}(K)) \le M \beta_Q \vol\big(g(Q)\cap K\big),
\end{align} 
here the volume is computed by fixing K\"ahler metrics on $Y,Z.$
\end{lemma}

\proof  Let $\omega$ be a K\"ahler metric on $Y.$ We obtain the  induced metric on $Q.$ Observe that since $Q$ is irreducible and $g$ is proper, $g(Q)$ is an irreducible analytic subset of $Z.$  By Hironaka's desingularisation theorem, there is  a proper smooth modification $\pi':Y' \to Y$ such that  the strict transform $Q'$ of $Q$ is smooth. Let $g':= g|_Q \circ \pi': Q' \to g(Q).$ Outside a proper analytic subset of $g(Q),$ the map $g'$ is a proper submersion. By applying   Ehresmann's lemma  (see for example \cite[Le. 10.2]{Demailly-theorieHodge}) to $g'$ and using the irreducibility of $g'(Q')$, we  see that generic fibers of $g'$ are (locally) $\cali{C}^\infty$-homotopic to each other. This combined with the fact that a generic fiber of $g|_Q$  is the direct image of a generic fiber of $g'$ implies that generic fibers of $g|_Q$ is (locally) $\cali{C}^\infty$-homotopic as well. Using this, the closedness of $\omega$ and the fact that every analytic set is naturally oriented,  we obtain that generic fibers of $g$ have the same volume denoted by $\beta_Q$ as in the statement. 

Since $g$ is a submersion, $Y$ is a fibered bundle over $Z.$  Consider a finite covering  $(Y_j)_{1 \le j \le N}$ of $Y$ over  a relatively compact open neighborhood of $K$ such that for every $j,$  $Y_j$ is a local (differentiable) trivialization of $Y,$ \emph{i.e,}  $Y_j= Z_j \times Y'_j,$ where $Z_j\subset Z$ and $Y'_j$ is diffeomorphic to a fiber of $g.$  We can also assume that $Y_j$ is a relatively compact open subset of another trivialization of $Y.$ Put $Q_j:= Q \cap Y_j.$ Consider a Riemannian metric on $Y_j$ which is the product of two Riemannian metrics on $Z_j, Y'_j.$ Using this metric and the fact that the volume of a generic fiber of $g|_{Q_j}$ is $\le  \beta_Q,$ we have 
$$\vol(Q_j) \lesssim \beta_Q  \vol(Z_j\cap K) \le \beta_Q \vol\big(g(Q)\cap K\big).$$      
Summing over $1\le j\le N$ in the last inequality gives 
$$\vol(Q\cap g^{-1}(K)) \le M \beta_Q \vol\big(g(Q)\cap K\big)$$
for some constant $M$ independent of $Q$. The other desired inequality is proved in the same way. This finishes the proof.  
\endproof

Let $\omega, \widehat \omega, \widehat \omega_h$ be as in the last section.

\begin{lemma} \label{le_isolatedpoint} Assume that $V_1, V$ are of complementary dimensions.  The following three properties hold:

$(i)$  for any isolated point $x$ in $V_1 \cap V$,   the multiplicity $\nu_x$ defined above is equal to the usual multiplicity of $x$ in the intersection $V_1 \cap V.$ Moreover, the only irreducible component $\widehat W$ of $\widehat V_1 \cap \widehat V$ such that $x \in \sigma(\widehat W)$ is $\sigma^{-1}(x),$

$(ii)$  for every positive $(l_1-1,l_1-1)$-form $\Phi$ on $\widehat X$ whose restriction to each fiber of  $\widehat V \approx \P(E)$ is of mass $1$ on that fiber, we have
\begin{align}\label{ine_MKisolate}
\sum_{x \in V_1 \cap V  \text{ isolated or exotic}} \nu_x \delta_x \le  \sigma_*\big( [\widehat V_1] \wedge [\widehat V] \wedge \Phi\big),
\end{align}

$(iii)$  if $V$ is compact and $\Phi$ is a form as in $(ii)$, then there exists a constant $M$ independent of $V_1$ such that 
\begin{align}\label{ine_MKisolate2new}
\tilde{P} \ge \sigma_*\big( [\widehat V_1] \wedge [\widehat V] \wedge \Phi\big)- M \sum_{j=0}^{l_1-2}\|\sigma_*\big( [\widehat V_1] \wedge [\widehat V] \wedge \widehat \omega^j\big)\|.
\end{align}
\end{lemma}


\proof  Let $T_\infty$ be the tangent current of $T:= [V_1]$ along $V.$ Let $x$ be an isolated point in the intersection $V_1 \cap V$ and $\nu_x'$ its multiplicity defined in the classical sense.   It is already observed in \cite[Le. 2.2]{DNT_equi} that in a small enough local chart around $x$ we have  
$$T_\infty= \nu_x' [\pi^{-1}(x)],$$
where $T_\infty$ is the tangent current of $T$ along $V.$  This can be seen  directly from the classical definition of the multiplicity of $x$.  Since $T_\infty= \pi_\infty^*([\widehat V_1] \wedge [\widehat V]),$  we deduce that $\nu'_x= \nu_x.$ Assertion $(i)$ follows immediately. 

We now prove $(ii)$. Note that
$$[\widehat V_1] \wedge [\widehat V]= \sum_{x \in V_1 \cap V \text{ isolated or exotic}}  \nu_x [\sigma^{-1}(x)]+ \sum_{\widehat W:  \dim \sigma(\widehat W)\ge 1}\alpha_{\widehat W}[\widehat W].$$
This implies that 
$$\sum_x  \nu_x [\sigma^{-1}(x)] \le \widehat V_1 \wedge \widehat V,$$
where the sum is taken over isolated points and exotic points  in $V_1 \cap V.$ The inequality (\ref{ine_MKisolate}) follow immediately.  

Assume now $V$ is compact. Cover $V$ with a finite number of local charts $V'.$ Thus $(\sigma|_{\widehat V})^{-1}(V')$ is K\"ahler. For $\widehat W \in \mathcal{W},$ denote by $\beta_{\widehat W}$ the volume of a generic fiber of $\sigma|_{\widehat W}.$  Applying Lemma \ref{le-submer-estimat} to the submersion $\sigma: (\sigma|_{\widehat V})^{-1}(V') \to V',$ we obtain that   
$$\vol(\widehat W) \lesssim \beta_{\widehat W} \vol\big(\sigma(\widehat W)\big)\lesssim \big \langle [\widehat W], \sigma^*\omega^{\dim \sigma(\widehat W)} \wedge \widehat \omega_h^{l_1-1-\dim \sigma(\widehat W)}\big \rangle$$
which implies 
$$\sum_{\widehat W: \dim \sigma(\widehat W)\ge 1}\alpha_{\widehat W}\vol(\widehat W) \le  \sum_{j=0}^{l_1-2}\|\sigma_*\big( [\widehat V_1] \wedge [\widehat V] \wedge \widehat \omega^j\big)\|.$$
The desired inequality (\ref{ine_MKisolate2new}) then follows.  The proof is finished.     
\endproof 

\begin{proposition} \label{pro_tontaiWnga2}
 Let $K,K'$ be  compact subsets of $X$ such that $K$ is contained in the interior of $K'.$  Assume that there exists a Hermitian form $\omega$ on $X$ for which $\ddc \omega^j=0$ on $V$ for $1\le j \le q-1,$ where $q= \dim V_1,$   there exists a constant $c$ independent of $V_1,$  we have    
\begin{align}\label{ine-excess-formua}
\sum_{W_1 \in \tilde{\mathcal{W}}_1}  \nu_{W_1} \vol(W_1 \cap K) \le  c \vol(V_1\cap K').
\end{align}
\end{proposition}

\proof  We will prove the following inequality:  there is a constant $c$ independent of $V_1$ for which  
\begin{align} \label{ine-excess-formua2}
\sum_{W_1 \in \tilde{\mathcal{W}}_1}\nu_{W_1} \vol(W_1 \cap K) \le  c \| \widehat V_1 \wedge \widehat V\|_{\sigma^{-1}(K')}.
\end{align}
Our desired inequality is deduced directly from (\ref{ine-excess-formua}) and Theorem \ref{th_Hhold}.     Let $\omega$ be a Hermitian metric on $X$ and  $\widehat \omega, \widehat \omega_h$  as in the last section.  
By Lemma \ref{le_tontaiWnga} and the definition of $\tilde{\mathcal{W}}_1$, there is  $\widehat W_1\in \mathcal{W}$ with $\sigma(\widehat W_1)= W_1$ and there are $\nu_{W_1}$ such $\widehat W_1$.  

Put $l_1':= \dim W_1.$ For $x \in W_1,$  let $F_x$ be the fiber over $x$ of $\widehat W_1 \to W_1.$ We have  
$$\dim F_x \ge (l_1-1- l'_1)$$
and the equality occurs for $x$ in  some open Zariski subset $W'_1$ of $W_1$.    Using 
$$\widehat \omega^{l_1-1} \gtrsim  \omega^{l'_1} \wedge  \widehat \omega_h^{l_1-1-l'_1}$$
gives 
$$\vol(\widehat W_1 \cap \sigma^{-1}(K)) \gtrsim \int_{\widehat W_1 \cap \sigma^{-1}(K)} \omega^{l'_1} \wedge  \widehat \omega_h^{l_1-1-l'_1}= \int_{x \in W'_1\cap K} \omega^{l'_1} \int_{F_x} \widehat \omega_h^{l_1-1-l'_1}$$
The second integral in the right-hand side of the last equality is equal to the cup product of the cohomology classes of $F_x$ and $(\widehat \omega_h|_{\sigma^{-1}(x)})^{l_1-1-l'_1}$ in $\sigma^{-1}(x) \approx \P^{l-1}$ which is thus  $\ge c_0$ for some strictly positive constant $c_0$ independent of $V_1$.  It follows that 
$$\vol(\widehat W_1 \cap \sigma^{-1}(K)) \gtrsim \int_{W_1\cap K} \omega^{l'_1}= \vol(W_1\cap K).$$
Consequently, 
$$\sum_{W_1 \in \tilde{\mathcal{W}}_1}\nu_{W_1} \vol(W_1 \cap K) \le  \sum_{\widehat W \in \mathcal{W}} \vol(\widehat W \cap \sigma^{-1}(K)) \le \|\widehat V_1 \wedge V\|_{\sigma^{-1}(K)}.$$
This finishes the proof. 
\endproof



The next result gives an inequality similar to (\ref{ine-excess-formua}) in a more particular situation where we only have a weaker assumption on $\omega.$

\begin{proposition} \label{pro_graph} Let $X$ be a compact complex surface, $Y$ and $Z$ two compact complex manifolds. Assume that $Y,Z$ admit Hermitian pluriclosed metrics.  Let $\Delta$ be the diagonal of $X^2$ and $\Delta_2:= Y \times \Delta \times Z.$ Let $V_1, V_2$ be  complex analytic subsets of dimension $2$ of  $Y \times X,$ $X \times Z$ respectively.  If $W_1, \ldots, W_m$ are the  irreducible components of  $(V_1 \times V_2) \cap \Delta_2,$ 
then we have
\begin{align}\label{ine_Wjvol}
\sum_{j=1}^m \vol(W_j) \le c_X \vol(V_1) \vol(V_2),
\end{align}
for some constant $c_X$ depending only on $X.$
\end{proposition}

Recall that a Hermitian metric $\omega$ is pluriclosed if $\ddc \omega=0.$

\proof  Since $X,Y,Z$ admits Hermitian pluriclosed metrics,    there exist  positive definite Hermitian forms $\omega, \omega_Y, \omega_Z$ on $X,Y,Z$ respectively such that $\ddc \omega= \ddc \omega_Y = \ddc \omega_Z=0.$  Put $X_2:= Y \times X \times X \times Z.$     Let $\sigma_2: \widehat{X}_2 \to X_2$ be the blowup of $X_2$ along $\Delta_2$ and $\sigma: \widehat{X\times X} \to X\times X$ the blowup of $X\times X$ along $\Delta.$ We see that $\widehat{X}_2= Y \times \widehat{X \times X} \times Z$ and $\sigma_2= (\id_Y, \sigma, \id_Z)$ because of the choice of $\Delta_2.$ 

Let $\widehat{\Delta}$ be the exceptional hypersurface of $\sigma$ and $\widehat{\Delta}_2$ the exceptional hypersurface of $\sigma_2.$ By the above observation, $\widehat{\Delta}_2= Y \times \widehat{\Delta} \times Z.$  Let $\widehat{\omega}_h$ be the Chern form of a Hermitian metric on $\mathcal{O}(-\widehat{\Delta})$ whose restriction to $\widehat{\Delta}$ is Fubiny-Study form on $\widehat{\Delta} \approx \P(N \Delta).$ Denote by $p_j$ the projection from $X_2$ to the $j^{th}$ component for $1 \le j \le 4.$ Put  
$$\omega_2:= p_1^* \omega_Y+ p_2^* \omega+p_3^* \omega+p_4^* \omega_Z.$$
By rescaling $\omega,$ we can assume that  $\widehat \omega:= \widehat \omega_h+ p_2^* \omega+p_3^* \omega >0.$ Hence  $\widehat{\omega}_2:= \widehat p^* \widehat{\omega}_h+ \sigma_2^* \omega_2>0$ as well, where $\widehat p$ is the natural projection from $\widehat{X}_2$ to $\widehat{X \times X}.$   

Theorem \ref{th_V1alongV} tells us that the tangent current to $T:= [V_1] \otimes [V_2]$ along $\Delta_2$ is unique and given by $\pi^*_\infty([\widehat T] \wedge [\widehat{\Delta}_2]),$ where $\widehat T$ is the strict transform of $T$ in  $\widehat{X}_2,$ and $\pi_\infty$ is the projection from $\P(N \Delta_2 \oplus \C)$ to  $\widehat{\Delta}_2 \approx  \P(N \Delta_2).$ On the other hand, Proposition \ref{pro_tontaiWnga2} implies that 
\begin{align}\label{ine_widehatrWj}
\sum_{j=1}^m \vol(W_j) \lesssim \int_{\widehat{X}_2} \widehat T\wedge [\widehat{\Delta}_2] \wedge \widehat{\omega}_2^{3}=:A. 
\end{align}
Thus in order to get (\ref{ine_Wjvol}), we only need to bound the last integral.  Denote by $(x_1,\ldots, x_4)$ a general point in $X^4.$ Write $[\widehat{\Delta}]= \ddc u+ \eta'$ for some smooth form $\eta'$ and some quasi-p.s.h. function $u$ on $\widehat{X \times X}.$ It follows that $[\widehat \Delta_2]= \ddc \widehat p^* u+ \widehat p^* \eta'.$ 

Let $\varphi, \eta$ be as in the proof of Theorem \ref{th_Hhold}, \emph{i.e,} $\sigma_* \widehat \omega_h= \ddc \varphi+ \eta.$ Note that  
$$\ddc (\varphi \circ \sigma) = \widehat \omega_h - \sigma^* \eta + c [\widehat \Delta]$$
 for some  strictly positive constant $c.$ 
Multiplying $\widehat \omega_h$ by $c^{-1}$ allows us to assume that $c=1.$ Hence, $\varphi \circ \sigma - u$ is a smooth function on $\widehat{X \times X}.$ This together with the Chern-Levine-Nirenberg gives 
\begin{align} \label{ine_Tutruvarphi}
\|  \widehat T \wedge \ddc(u \circ \widehat p - \widehat \varphi)\| \lesssim  \| \widehat T\|,
\end{align}
where $\widehat \varphi:=\varphi \circ \sigma \circ \widehat p$ and  we recall the wedge product in the last inequality is defined classically, \emph{i.e,}  $u$ (hence $\varphi \circ \sigma$) is integral with respect to $\widehat T.$ Let $c_1$ be a positive constant such that
\begin{align*} 
\eta' \le   c_1 \widehat \omega, \quad  \ddc (\varphi \circ \sigma)+ c_1 \widehat \omega \ge 0.
\end{align*}
Using   this, (\ref{ine_Tutruvarphi}) and the fact that $\widehat \varphi$ is integrable with respect to $\widehat T$, we see that 
\begin{align} \label{ine_uocluongA}
A  &\lesssim \|\widehat T\|+ \int_{\widehat{X}_2} \widehat T\wedge \big(\ddc \widehat \varphi+ c_1 \widehat \omega_2\big) \wedge \widehat{\omega}_2^3 \\
\nonumber
&= \|\widehat T\|+ \liminf_{M \to \infty}\int_{\widehat{X}_2} \widehat T\wedge \big(\ddc \widehat \varphi_M+ c_1 \widehat \omega_2 \big) \wedge \widehat{\omega}_2^3,
\end{align}
where $\widehat \varphi_M:= \max \{\widehat \varphi, -M\}.$  Denote by $A_M$ the last integral.   Since $\widehat T \wedge \ddc \widehat \varphi_M$ has no mass near $\widehat \Delta_2,$ we get
\begin{align} \label{eq_uocluongAm}
A_M &= \int_{X_2 \backslash \Delta_2}  T\wedge \big(\ddc \varphi_M+ c_1 (\sigma_2)_*\widehat \omega_2\big) \wedge \big((\sigma_2)_*\widehat{\omega}_2\big)^3.
\end{align}
Notice that $\varphi, \varphi_M$ are functions of $(x_2, x_3)$ and $(\sigma_2)_* \widehat \omega_2= \ddc \varphi+ \eta+ \sum_{j=1}^4 p_j^* \omega$  and $\eta$ is a closed smooth form. Let $c_2$ be a positive constant such that  $(c_2-c_1)\omega_2 \ge c_1 \eta.$  Thus for $\varphi'_{M_1}:= \max \{\varphi_M+ c_1\varphi, -M_1\},$ we have $\ddc \varphi'_{M_1}+ c_2\omega_2\ge 0.$ Using this and  (\ref{eq_uocluongAm})  gives 
\begin{align} \label{eq_uocluongAm2}
A_M \le  \liminf_{M_1 \to \infty} \int_{X_2}  T\wedge \Phi_{M_1}, 
\end{align}
where 
$$\Phi_{M_1}: =\big(\ddc \varphi'_{M_1}+ c_2 \omega_2\big) \wedge \big(\ddc \varphi_{M_1}+  c_2 \omega_2 \big)^3.$$
Put $\omega_{21}:= p_1^* \omega_Y,$ $\omega_{22}:= p_2^* \omega,$ $\omega_{23}:= p_3^* \omega$ and $\omega_{24}:= p_1^* \omega_Z.$ Since $\omega_2= \sum_{j=1}^4 \omega_{2j},$ we can write $\Phi_{M_1}$ as a linear combinations of forms
$$\Phi_{M_1; s',s, \mathbf{l}}:= (\ddc \varphi'_{M_1})^{s'} \wedge (\ddc \varphi_{M_1})^s \wedge \wedge_{j=1}^4 \omega_{2j}^{l_j}$$ 
with $\mathbf{l}=(l_1, \ldots, l_4)$ and  
$$s'+ s+l_1+l_2+l_3+l_4= 4, \quad  0 \le s' \le 1.$$
So to bound $A_M,$ we only need to bound $\langle T, \Phi_{M_1; s',s, \mathbf{l}}\rangle.$ If $s'+s=4$ or $3,$ then 
$$\langle T, \Phi_{M_1; s',s, \mathbf{l}}\rangle=0$$
 because of Stokes' theorem and $\ddc \omega=0.$ Recall that $T= [V_1] \otimes [V_2].$  If $s'+s=0,$ then $\langle T, \Phi_{M_1; s',s, \mathbf{l}}\rangle$ is bounded by the mass of $T.$ On the other hand, if $l_1+l_2=1,$ we can apply Stokes' theorem to  $\langle [V_1], \Phi_{M_1; s',s, \mathbf{l}} \rangle$  to show that this product is equal to $0$ because $\ddc (\omega_{21}^{l_1} \wedge \omega_{22}^{l_2})=0.$ Hence $\langle T, \Phi_{M_1; s',s, \mathbf{l}}\rangle=0.$ If $l_3+l_4=1,$ we obtain the same conclusion. So it remains to treat the case where $(l_1+l_2-1)(l_3+l_4-1) \not =0$ and $s'+s=1$ or $2.$     

We first consider $s'+s=1$.   We have $\sum_j l_j=3.$ Let $l'_1, \ldots, l'_4$ be the numbers $l_1, \ldots, l_4$ written in an order such that  $l'_1 \ge \cdots \ge l'_4.$   Hence $l'_4=0$ because otherwise $\sum_j l'_j \ge 4,$ a contradiction. We then see easily that  either $l'_1=l'_2=l'_3=1$ or $l'_3=0,$ $l'_2=1,$ $l'_1=2.$  The first case can't happen because otherwise we will get $(l_1+l_2-1)(l_3+l_4-1) =0$.
Hence, we obtain  $l'_3=0,$ $l'_2=1,$ $l'_1=2$ and  $l'_4=0.$ It follows that   $(l_2,l_3)= (0,1)$ or $(1,0)$ and $(l_1,l_4)=(0,2)$ or $(2,0)$ because $\varphi_{M_1}, \varphi'_{M_1}$ depends only on $x_2,x_3$ and $\dim X=2.$  Without loss of generality, we can suppose $l_4=2.$ This combined with the fact that $\dim V_1= \dim V_2=2$ gives
$$\langle T, \Phi_{M_1; s',s, \mathbf{l}}\rangle= \int_{V_2} \omega_4^2(x_4) \int_{V_1}(\ddc_{x_2} \varphi'_{M_1})^{s'} \wedge (\ddc_{x_2} \varphi_{M_1})^s \wedge \omega_{1}(x_1)^{l_2} \wedge \omega_2^{l_3}(x_2)=0$$ 
by Stokes' theorem and $l_2+l_3=1.$

We now consider $s'+s=2.$ We have $\sum_j l_j=2.$ Let $l'_1, \ldots, l'_4$ be as above. Arguing as above gives $l'_3=l'_4=0$ and $(l'_1,l'_2)=(2,0)$ or $(l'_1,l'_2)=(1,1).$ If the latter case  happens,  we get either $l_1=l_2=0, l_3=l_4=1$ or $l_1=l_2=1,$ $l_3=l_4=0.$ For these both cases, the stokes' theorem gives the $\langle T, \Phi_{M_1; s',s, \mathbf{l}}\rangle=0.$  The case where $(l'_1,l'_2)=(2,0)$ is treated similarly.

Hence, we have proved that  $\langle T\wedge \Phi_{M_1}\rangle$ is $\lesssim \|T\|$ independent of $M_1.$ Combining this with (\ref{eq_uocluongAm2}), (\ref{ine_uocluongA}) and (\ref{ine_widehatrWj}) gives the desired inequality. The constant $c_X$ depends only on $X$ because  all of constants in the estimates we used above do so. The proof is finished.
\endproof



\section{Number of isolated periodic points} \label{sec_mapdominant}


Let $X$ be a compact complex manifold of dimension $k$ and $f$ a (dominant) meromorphic self-correspondence of $X.$ In this section, we  estimate the number of isolated periodic points of $f.$ Let $\pi_1, \pi_2$ be the natural projections of $X\times X$ to the first and second components respectively.  

Recall that by definition, $f$ is given by  an effective analytic cycle $\Gamma:=\sum_{j} \Gamma'_j,$ where $\Gamma'_j$ are irreducible $k$-dimensional analytic subsets of $X\times X$ such that the images of $\Gamma_j$ under $\pi_1, \pi_2$ are equal to $X$. The cycle  $\Gamma$ is called \emph{the graph of $f.$} The \emph{adjoint correspondence} $f^{-1}$ of $f$ is the self-correspondence of $X$ whose graph is the image of $\Gamma$ by the involution of $X^2$ sending $(x,y)$ to $(y,x)$ for every $(x,y) \in X^2.$   We can still define the self-composition $f^n$ of $f$ ($n \in \N$) which is again a meromorphic self-correspondence on $X$ and  dynamical degrees $d_q(f)$ of $f$ as in Introduction; see \cite{Ds_upperbound_mero} for the K\"ahler case. Notice that $d_0(f), d_k(f)$ are equal to the topological degree of $\pi_1|_{\Gamma},$ $\pi_2|_{\Gamma}$ respectively. 

The \emph{indeterminancy set} of $f$ is defined by 
$$I(f):= \{x\in X: \dim \pi_1^{-1}(x)\cap \Gamma >1\}.$$ 
This is an analytic subset of codimension at least $2$ of $X.$  If $I(f)$ is empty, $f$ is called a \emph{holomorphic correspondence}.  By \cite[Le. 4.7]{DNV}, if $X$ is K\"ahler, for every holomorphic self-map $f$ of $X,$ $f^{-1}$ is a holomorphic correspondence.   

We will need the following lemma generalizing a similar inequality due to Dinh \cite{D_suite} in the K\"ahler case.

\begin{lemma} \label{le_spectralgap} Let $X$ be a compact complex manifold. Let $d_0, \ldots, d_k$ be the dynamical degrees of a meromorphic self-correspondence $f$ of $X.$ Then given every smooth $(p,q)$-form  $\Phi,$ we have 
\begin{align} \label{ine_spec}
\limsup_{n \to \infty} \| (f^n)^* \Phi \|^{1/n} \le  \sqrt{d_p d_{q}}, \quad \limsup_{n \to \infty} \| (f^n)_* \Phi \|^{1/n} \le  \sqrt{d_{k-p} d_{k-q}}.
\end{align}
\end{lemma}

\proof   Let $\omega$ be a Hermitian metric on $X.$ The second inequality of (\ref{ine_spec}) is a direct consequence of the first one by using $f^{-1}$ instead of $f.$  Let $\Phi$ be a smooth $(p,q)$-form.  Without loss of generality, we can suppose that $q\ge p.$ 

   By using a partition of unity, we can write $\Phi$ as a sum of forms of type $\Phi':=\Phi_{(p,p)} \wedge \Phi_{(0,q-p)}$ for some positive smooth form $\Phi_{(p,p)}$ of bidegree $(p,p)$ and some $(0,q-p)$-form $\Phi_{(0,q-p)}$. Let $\Psi$ be a smooth $(k-p, k-q)$-form. Similarly, we can write $\Psi$ as a sum of forms of type $\Psi':= \Psi_{(k-q,k-q)}\wedge \Psi_{(q-p,0)}$ for some  positive form $\Psi_{(k-q,k-q)}.$ It follows that in order to estimate $\big|\langle (f^n)^* \Phi, \Psi \rangle\big|$, it is sufficient to estimate $\big|\langle (f^n)^* \Phi', \Psi' \rangle\big|.$ 

Let $\pi_1, \pi_2$ be the natural projections from $X^2$ to the first and second components respectively.  Recall $(f^n)^* \Phi= (\pi_1)_* ([\Gamma_n] \wedge \pi_2^* \Phi),$ where $\Gamma_n$ is the graph of $f^n.$ Thus, 
$$\big\langle (f^n)^* \Phi', \Psi' \big\rangle= \int_{\Gamma_n} \pi_2^* \Phi \wedge \pi_1^* \Psi'.$$
By the Cauchy-Schwarz inequality, we have  
\begin{multline*}
\big|\big\langle (f^n)^* \Phi', \Psi' \big\rangle\big| \le  \bigg(\int_{\Gamma_n} \pi_2^*(\overline \Phi_{(0,q-p)} \wedge \Phi_{(0,q-p)}) \wedge  \pi_2^*(\Phi_{(p,p)}) \wedge \pi_1^*\Psi_{(k-q,k-q)}\bigg)^{1/2}\\
\bigg(\int_{\Gamma_n} \pi_1^*(\Psi_{(q-p,0)} \wedge \overline \Psi_{(q-p,0)}) \wedge  \pi_2^*(\Phi_{(p,p)}) \wedge \pi_1^*\Psi_{(k-q,k-q)}\bigg)^{1/2}
 \end{multline*}
which is $\lesssim$
$$ \bigg(\int_{\Gamma_n} \pi_2^*\omega^{q-p} \wedge  \pi_2^*\omega^p  \wedge \pi_1^*\omega^{k-q}\bigg)^{1/2}
\bigg(\int_{\Gamma_n} \pi_1^*\omega^{q-p} \wedge  \pi_2^*(\omega^{p}) \wedge \pi_1^*\omega^{k-q}\bigg)^{1/2}.$$
Hence, the desired inequality follows. The proof is finished.
\endproof


Let $\Gamma_n$ be the graph of $f^n$ and $\Delta$ the diagonal of $X^2.$ Let $\sigma:\widehat{X \times X} \to X\times X$ be the blowup of $X^2$ along $\Delta.$ Denote by  $\widehat \Gamma_n$ the strict transform of $\Gamma_n$ via $\sigma.$ Let $\widehat \Delta$ be the exceptional hypersurface. Let $\sigma|_{\widehat \Gamma_n \cap \widehat \Delta}$ be the restriction of $\sigma$ to $\widehat \Gamma_n \cap \widehat \Delta.$

\begin{definition} \label{def-exotic-periodic} \emph{An exotic (non-isolated) periodic point $a$ of period $n$} of $f$ is  an exotic intersection point  of  $\Gamma_n$ and  $\Delta,$ this means that $a$ is a non-isolated point in the set  $\Gamma_n \cap \Delta$ and  $\dim (\sigma|_{\widehat \Gamma_n \cap \widehat \Delta})^{-1}(a)= k-1$ (the maximal possible dimension).  The multiplicity of $a$ is that of $a$ as an exotic intersection point of $\Gamma_n$ and $\Delta,$ see Section \ref{sec_analy}.  
\end{definition}
Let $\tilde{P}_n$ be the sum of $P_n$ and the number of exotic periodic points of period $n$ counted with multiplicity.   
Let $\cali{G}$ be the set of  compact complex manifolds $X$ possessing a Hermitian metric $\omega$ such that $\ddc \omega^j= 0$ for $1 \le j \le k-1,$ where $k:= \dim X.$ The Hermitian metric $\omega$ with the last properties  has been studied by Fino-Tomassini in \cite{Fino_Tomassini_anestho,Fino_Tomassini_blowup}. This notion  is related to  the anestho-K\"ahler metric  introduced  by Jost-Yau in \cite{Jost_Yau}  and strong KT metrics surveyed in \cite{F-T-survey}.   

Clearly, $\cali{G}$ contains every compact K\"ahler manifold.  By a result of Gauduchon \cite{Gauduchon}, every $k$-dimensional compact complex manifold admits a Gauduchon metric $\omega,$ \emph{i.e,} $\omega$ is a Hermitian metric with $\ddc \omega^{k-1}=0.$  Hence \emph{every compact complex surface belongs to} $\cali{G}.$ We refer to \cite{Fino_Tomassini_anestho} for more examples of manifolds in $\cali{G}.$ We remark however that Hopf manifolds of dimension $>2$ is not in $\cali{G}$, see \cite[Th. 8.3]{Egidi}.  

We would like to emphasize a key difference between $\cali{G}$ and the class of K\"ahler manifolds is that in contrast to the K\"ahler case,  \emph{we don't know whether the product of two manifolds in $\cali{G}$ is in $\cali{G}.$} It is very likely that this is not true, see \cite{Fino_Tomassini_anestho} for some related comments.  That problem is a crucial difficulty when studying the dynamics of self-maps on $X \in \cali{G}$  because in order to study dynamical properties of self-maps of  $X,$ we often have to work on the Carterisan products $X^n$ of $X$ with $n \in \N.$  Theorem \ref{th_main2phay} is a direct consequence of the following result. 

\begin{theorem} \label{th_main2phayG} Let $X\in \cali{G}$  and $f$ a dominant meromorphic self-map of $X.$ Then we have
\begin{align}\label{ine_upperboundPnG}
\limsup_{n \to \infty} \frac{1}{n} \log \tilde{P}_n \le h_a(f).
\end{align}
When $X$ is of dimension $2,$ then the algebraic entropy $h_a(f)$ of $f$ is a finite  bi-meromorphic invariant of $f$ and 
\begin{align}\label{ine_upperboundhfG}
h_t(f) \le h_a(f)<\infty.
\end{align}
\end{theorem}

\proof By the hypothesis, there is a Hermitian metric $\omega$ on  $X$ with $\ddc \omega^j=0$ for $1\le j \le k-1.$ Let $\omega_2:= p_1^* \omega+ p_2^* \omega$ where $p_1,p_2$ are the projections from $X^2$ to the first and second components respectively. Let $\Delta$ be the diagonal of $X^2.$  Observe that $\ddc \omega_2^j=0$ on $\Delta$ for $1 \le j \le k-1.$    Let $\Gamma_n$ be the graph of $f^n$ on $X^2.$ Observe that $[\Gamma_n]$ is a closed positive current of bidimension $(k,k)$ on $X^2.$  Applying Proposition \ref{pro_tontaiWnga2} to $q=k;$ $X^2$ in place of $X,$ $V_1:= \Gamma_n, V:= \Delta$ and $K:=X^2,$ we obtain that 
$$\tilde{P}_n   \lesssim \, \vol(\Gamma_n).$$
This combined with the fact that $\lim_{n\to \infty} [\vol(\Gamma_n)]^{1/n} e^{- h_a(f)}=1$ gives (\ref{ine_upperboundPnG}).

Now assume that $X$ is a compact complex surface.  We will prove that $h_a(f)$ is finite. To this end, we need to estimate $\vol(\Gamma_n).$ Let $n_1, n_2$ be positive integers.   Put $V_j:= \Gamma_{n_j}$ for $j=1,2.$   Consider  the intersection $(V_1 \times V_2) \cap (X \times \Delta \times X)$ in $X^{4}.$ Let $p_{1,4}$ be the projection from $X^{4}$ to $X^2$ by sending $(x_1, \ldots,x_4)$ to $(x_1,x_4).$ By the definition of $\Gamma_{n_1+n_2},$  there exists a $k$-dimensional irreducible component $W$ of $(V_1 \times V_2) \cap (X \times \Delta \times X)$   such that $\Gamma_{n_1+n_2}= p_{1,4}(W).$ 
Using  $\dim \Gamma_{n_1+n_2} = \dim W,$ we have $\vol(\Gamma_{n_1+n_2}) \le \vol(W)$ which is 
$$ \le c_X \vol(\Gamma_{n_1}) \vol(\Gamma_{n_2})$$
by Proposition \ref{pro_graph}. 
It follows that $\limsup_{n\to \infty} [\vol(\Gamma_n)]^{1/n}$ exists  and is a finite number. On the other hand, we can check directly that
$$\max\big\{ \int_X (f^n)^* \omega^q \wedge \omega^{k-q}: 0 \le q \le k \big\} \lesssim  \vol(\Gamma_n) \lesssim \max\big\{ \int_X (f^n)^* \omega^q \wedge \omega^{k-q}: 0 \le q \le k \big\}.$$ 
Thus, $h_a(f)=\limsup_{n\to \infty} [\vol(\Gamma_n)]^{1/n} < \infty.$     

Now consider a bi-meromorphic map $g: X \to X'$ and $f':= g\circ f \circ g^{-1}: X'\to X'.$ We need to show that $h_a(f')= h_a(f).$ Observe that $f'^n= g \circ f^n \circ g^{-1}.$ Applying similar arguments as above gives $$\vol(\Gamma_{n}) \lesssim \vol(\Gamma'_{n}) \lesssim \vol(\Gamma_n),$$
where $\Gamma'_{n}$ is the graph of $f'^n.$  Consequently, $h_a(f')=h_a(f).$ In other words, $h_a(f)$ is a bi-meromorphic invariant of $f.$  

It remains to prove (\ref{ine_upperboundhfG}).  Let $\Gamma_{[n]}$ be the graph of  $(f, f^2 \ldots, f^n)$ in $X^{n+1}.$ For  $1 \le s\le k$ and   $M=(n_1, \ldots, n_{s})$ in $\N^{s}$ with $1 \le n_1< \cdots <n_s \le n,$ denote by $\Gamma_M$ the image  of the map $(f^{n_1}, \ldots, f^{n_s})$ in $X^{s}.$     It was proved in  \cite{Gromov_entropy,Ds_upperbound_mero} that 
$$h_t(f) \le  \lov(f):= \limsup_{n\to \infty} [\vol(\Gamma_{[n]})]^{1/n}.$$
Using an appropriate metric on $X^n$ induced from that on $X,$ we can see that 
$$\vol(\Gamma_n) \lesssim  \sum_{M} \vol(\Gamma_M),$$
where the sum is taken over $M= (n_1, \ldots, n_k)$ with $0 \le n_1< \cdots < n_k \le n.$
Since the number of such $M$ is $\le n^k,$ in order to get the desired bound for $\lov(f),$ we only need to bound $\vol(\Gamma_M).$ Fix a such $M=(n_1, \ldots, n_k).$ Recall $k=\dim X=2.$ Thus,
$$\vol(\Gamma_M)\lesssim \sum_{0 \le q \le 2}  \int_{X}(f^{n_1})^* \omega^{q} \wedge (f^{n_2})^* \omega^{k-q} \lesssim  \sum_{0 \le q \le 2}  \int_{X}(f^{n_1})^*\big( \omega^{q} \wedge (f^{n_2-n_1})^* \omega^{k-q}\big).$$
The last term in the above inequality is equal to 
$$d_2(f)^{n_1} \int_X \omega^{q} \wedge (f^{n_2-n_1})^* \omega^{k-q} \le [h_a(f)+\epsilon]^{n_1}  [h_a(f)+ \epsilon]^{n_2-n_1} \le [h_a(f)+ \epsilon]^{n}$$
for any constant $\epsilon >0$ 	and $n\ge n_\epsilon.$ Therefore, we get
$$\lov(f) \le h_a(f).$$
A direct computation shows that $\lov(f) \ge h_a(f).$ It follows that $\lov(f)=h_a(f).$  This finishes the proof.
\endproof

Dinh-Sibony \cite{Ds_upperbound_mero} defined the topological entropy of a meromorphic correspondence. Using this definition, one can see that Theorem \ref{th_main2phayG} still holds for meromorphic correspondences. 

Assume now $X \in \cali{G}.$ Let $\omega'$ be a Hermitian metric on $X$ for which $\ddc \omega'^j=0$ for $1 \le j \le k-1$.   This metric induces naturally a metric $\omega:= \pi_1^* \omega'+ \pi_2^* \omega'$ on $X \times X$ with $\ddc \omega^j=0$ on $\Delta$ for $1 \le j \le k-1$.    

Let $\widehat \omega_h$ be a Chern form of a Hermitian metric of $\mathcal{O}(-\widehat \Delta)$ whose restriction to  each fiber of the projection $\widehat \Delta \to \Delta$ is a strictly positive and belongs to the cohomology class of a hyperplane of that fiber.  By rescaling $\omega',$ we can assume that   $\widehat \omega:=  \sigma^* \omega+ \widehat \omega_h>0.$ By our choice of $\widehat \omega_h,$ the restriction of  $\widehat\omega^{k-1}$ to each fiber of the projection $\widehat \Delta \to \Delta$ is a volume form of mass $1.$

\begin{proposition} \label{pro_equidistrG} Let $X\in \cali{G},$ $\omega, \widehat \omega$ as above and $f$ a meromorphic self-correspondence of $X.$  Assume that there are a current $T_\infty$ and  a sequence $(A_n)_{n\in \N}$ of  strictly positive numbers such that the following conditions are satisfied:

$(i)$ $\widehat T_\infty \wedge [\widehat \Delta]$ is well-defined classically,

$(ii)$ $A_n^{-1}[\Gamma_n]$ converges to $T_\infty$,

$(iii)$
\begin{align}\label{ea-dkTinftyDelta} \langle \widehat T_\infty \wedge [\widehat \Delta], \widehat \omega^{k-1}\rangle=1, \quad \sigma_*(\widehat T_\infty \wedge [\widehat \Delta] \wedge \widehat \omega^{j})=0
\end{align}
for every $0\le j \le k-2.$ 

Then we have 
$$\tilde{P}_n= A_n +o(A_n)$$
as $n\to \infty.$  

Moreover the last asymptotic still holds without  Assumption $(i)$ provided that  $\omega$ is  K\"ahler and  Assumption $(iii)$ is replaced by  the condition that the mass of the total tangent class  $\kappa^{\Delta}(T_\infty)$ of $T_\infty$ along $T$ measured with respect to $\widehat \omega$ is $1$ and the h-dimension of $\kappa^{\Delta}(T_\infty)$ is $0.$  
\end{proposition}

\begin{proof} Assumption $(iii)$ tells us that the h-dimension of $\widehat T_\infty \wedge [\widehat \Delta]$ is $j_\infty:=0.$ Since $\dim \Gamma_n+ \dim \Delta= 2k= 2k + j_\infty,$  applying Proposition \ref{pro_semicontinuity2} to $T_n:= A_n^{-1} [\Gamma_n], T_\infty,$ $V:= \Delta$ and $X^2,$ we obtain 
\begin{align} \label{conver_TVgamman}
\lim_{n\to \infty}\big\|\sigma_* (\widehat T_n \wedge [\widehat \Delta] \wedge \widehat \omega^{k-1}) \big\|= \big\|\sigma_*\big(\widehat T_\infty \wedge [\widehat \Delta] \wedge \widehat \omega^{k-1}\big)\big\|
\end{align}
for $0 \le  j\le k-1$. This combined with (\ref{ine_MKisolate}) of Lemma \ref{le_isolatedpoint} implies 
\begin{align}\label{ine_MKisolate2}
\limsup_{n\to \infty} A_n^{-1}  \tilde{P}_n \le   \big\| \sigma_*\big(\widehat T_\infty \wedge [\widehat \Delta] \wedge \widehat \omega^{k-1}\big)\big\|=\langle \widehat T_\infty \wedge [\widehat \Delta], \widehat \omega^{k-1}\rangle=1
\end{align}
by the first equality of  (\ref{ea-dkTinftyDelta}).  The second equality of (\ref{ea-dkTinftyDelta}) together with Property $(iii)$ of Lemma  \ref{le_isolatedpoint} applied to $V_1= \Gamma_n,$ $V=\Delta$ yields 
$$\liminf_{n\to\infty}A_n^{-1} \tilde{P}_n \ge 1.$$
We conclude that $\tilde{P}_n= A_n+ o(A_n).$    When $\omega$ is K\"ahler, similar arguments give the desired assertion by using Proposition \ref{pro_continui-totaltangenclass}  instead of Proposition \ref{pro_semicontinuity2}.  This finishes the proof.
\end{proof}

We recall here the following result from \cite{Vu_nonkahler_topo_degree} which is stated for meromorphic self-map but its proof is extended obviously to the case of self-correspondence. 

\begin{theorem} \label{th_dominant} Let $X$ be a compact complex manifold of dimension $k$ and $f$ a meromorphic self-correspondence of $X$. Let $\nu$ be a complex measure with $L^{k+1}$ density on $X$ and $\nu(X)=1.$  Assume that $d_k> d_{k-1}.$ Then the sequence $d_k^{-n}(f^n)^* \nu$ converges to  an invariant $PC$ probability measure $\mu_f$ of entropy $\ge \log d_k$ independent of $\nu$ such that  $d_k^{-1} f^* \mu_f= \mu_f$  and   
\begin{align} \label{conver_psh}
\lim_{n \to \infty} \langle d_k^{-n}(f^n)^* \nu - \mu_f, \varphi \rangle =0
\end{align}
for every quasi-p.s.h. function $\varphi$ on $X.$
\end{theorem}

The following result implies Theorem \ref{the_equidistr} in Introduction. 

\begin{theorem} \label{the_equidistrG} Let $X$ be a compact complex manifold and $f$ a meromorphic self-correspondence of $X.$   Assume that one of the following situations occurs:

$(i)$  $X\in \cali{G}$ and $f$ has a dominant topological degree,

$(ii)$ $X$ is K\"ahler and $f$ is a holomorphic self-correspondence of $X$ such that $f^{-1}$ is also a holomorphic correspondence and $f$ has a simple action on the cohomology groups,

$(iii)$ $X$ is a compact K\"ahler surface and $f$ is an algebraically stable dominant meromorphic self-map of $X$ with a  minor (small) topological degree.

Then  we have 
\begin{align}\label{eq-Pnngafpeirodic}
 \tilde{P}_n  =  e^{n h_a(f)}+ o(e^{n h_a(f)}).
\end{align}
\end{theorem}

Recall that if $f$ is a holomorphic self-correspondence of a compact K\"ahler manifold $X$, then  by \cite{DNV}, the action $f^*$ of $f$ on the cohomology is well-defined and $(f^n)^*=(f^*)^n$ for every $n.$ In this case, the dynamical degree $d_q$ of $f$ is thus equal to the spectral radius of $f^*|_{H^{q,q}(X)}$ and  we say that $f$ has   \emph{a simple action on the cohomology groups}  if   the maximal dynamical degree $d_p(f)$ of $f$ is a simple eigenvalue of the action $f^*$ on $H^{p,p}(X)$ and is the only eigenvalue of $f^*$ on $H^{p,p}(X)$ of modulus $d_p$ and for every $q\not =p$ we have $d_q(f)<d_p(f).$ Notice that by \cite{Diller-Farve}, every map in the situation $(iii)$ above also satisfies the last property.

\proof  
In order to prove (\ref{eq-Pnngafpeirodic}), we only need to check the assumption in Proposition \ref{pro_equidistrG}. Consider first the case where $f$ has a dominant topological degree and $X \in \cali{G}.$   Recall that smooth forms on $X^2$ can be approximated by forms $\pi_1^* \Phi_1 \wedge \pi_2^* \Phi_2$ in $\cali{C}^\infty$-topology, where $\Phi_1, \Phi_2$ are smooth forms on $X.$  By Lemma \ref{le_spectralgap}, we see that for smooth $(k-p,k-q)$-form $\Phi_1$ and $(p,q)$-form $\Phi_2$ on $X,$  
\begin{align} \label{limit_tinhpdGamma}
\big \langle d_k^{-n}[\Gamma_n], \pi_1^* \Phi_1 \wedge \pi_2^* \Phi_2 \big \rangle= d_k^{-n}\int_X (f^n)^* \Phi_2 \wedge \Phi_1 =O((d_p d_q)^{n/2}d_k^{-n}) \to 0
\end{align}
if $(p,q)\not = (k,k)$ because $d_k > d_p$ or $d_q$ in this case.  Let $\mu_f$ be the equilibrium measure of $f.$      On the other hand, if $(p,q)=(k,k),$ we have 
$$\big \langle d_k^{-n}[\Gamma_n], \pi_1^* \Phi_1 \wedge \pi_2^* \Phi_2 \big \rangle= \big \langle d_k^{-n}(f^n)^* \Phi_2, \Phi_1 \big \rangle    \to  \langle \mu_f, \Phi_1 \rangle \int_X \Phi_2= \langle \pi_1^* \mu_f, \pi_1^* \Phi_1 \wedge \pi_2^* \Phi_2 \rangle$$
by Theorem \ref{th_dominant}. Using this and (\ref{limit_tinhpdGamma}) gives 
$$d_k^{-n}[\Gamma_n] \to T_\infty:=\pi_1^* \mu_f.$$
Put $A_n:= d_k^{n}.$    Denote by $\Pi_j: \widehat{X\times  X} \to X$ the composition of $\sigma$ and $\pi_j$ for $j=1,2.$ Observe that $\Pi_1, \Pi_2$ are  submersions. Consider local coordinates $x=(x_1, \ldots,x_k)$ on $X.$ These coordinates induce a natural coordinate system $(x,y)$ on $X^2.$ The diagonal $\Delta$ is given by $x-y=0.$ Put $y':= x-y.$ We obtain new local coordinates $(x,y').$ A typical local chart on $\widehat{X\times X}$ can be described as $(x, y'_1, v_2, \ldots, v_k)$ and 
$$\sigma(x,y'_1, v_2, \ldots, v_k)= (x, y'_1 v_2, \ldots, y'_1 v_k), \quad \widehat \Delta= \{y'_1=0\}$$
We deduce that the fiber of $\Pi_1$ in the considered local chart is parameterized by $y'_1,$ $v_2, \ldots, v_k.$ Since $\widehat T_\infty=\sigma^* \pi_1^* \mu_f= \Pi_1^* \mu_f,$ we can check easily that $\widehat T_\infty\wedge [\widehat \Delta]$ is well-defined classically and  
\begin{align}\label{eqtinhTinftyomega}
\langle \widehat T_\infty \wedge [\widehat \Delta],   \widehat \omega^{k-1} \rangle &= \int_{x \in X} d\mu_f \int_{(y'_1,v_2,\ldots, v_k) \in \Pi_1^{-1}(x)}  [\widehat \Delta] \wedge \widehat \omega^{k-1}  \\
\nonumber
&=  \int_{x \in X} d\mu_f \int_{\sigma^{-1}\{(x,x)\}} \widehat \omega^{k-1}\\
\nonumber
&=\int_{x \in X} d\mu_f \int_{\sigma^{-1}\{(x,x)\}} \widehat \omega_h^{k-1} =\int_{x \in X} d\mu_f=1.    
\end{align}
Using similar computations as in (\ref{eqtinhTinftyomega}) gives  
$$\langle \widehat T_\infty \wedge [\widehat \Delta],   \widehat \omega^{j}\wedge \sigma^* \omega^{k-1-j} \rangle=0$$
for every $0\le j \le k-2.$ 

Consider now the case where $f$ has a simple action on the cohomology groups and $X$ is K\"ahler.    It was shown in \cite{DNV} that $T_n:=d_p^{-n}[\Gamma_n]$ converges to $T_\infty:=T^+\otimes T^{-},$ where $T^+, T^{-}$ are closed positive currents of bi-degree $(p,p)$ and $(k-p,k-p)$ respectively such that their super-potentials are continuous and $T^+ \wedge T^{-}$ is of mass $1$. We deduce that $\Pi_1^* T^+, \Pi_2^* T^{-}$ also have continuous super-potentials, see \cite[Le. 2.2]{DNV}. Thus the current $\Pi_1^* T^+ \wedge \Pi_2^* T^{-}$ is well-defined and has a continuous super-potential (cf. \cite[Pro. 3.3.3]{DS_superpotential}). Recall that a closed positive current with continuous super-potential has no mass on pluripolar set. It follows that  $\Pi_1^* T^+ \wedge \Pi_2^* T^{-}$ has no mass on $\widehat \Delta.$ Outside $\widehat \Delta$ the current $\Pi_1^* T^+ \wedge \Pi_2^* T^{-}$ is equal to $\widehat T_\infty.$ We then obtain that $\widehat T_\infty= \Pi_1^* T^+ \wedge \Pi_2^* T^{-}.$ Consequently, $\widehat T_\infty\wedge [\widehat \Delta]$ is well-defined classically. By regularizing $T^+, T^{-}$ in the SP-convergence (cf. \cite[Pro. 3.2.8]{DS_superpotential}), we get
$$\widehat T_\infty \wedge [\widehat \Delta]= (\sigma|_{\widehat \Delta})^*\big((T^+ \otimes T^{-}) \wedge \Delta\big)=(\sigma|_{\widehat \Delta})^*(T^+ \wedge T^{-}).$$
The first equality of (\ref{ea-dkTinftyDelta}) then follows. The second one is checked similarly. In fact, the above equality is more than needed to obtain  (\ref{ea-dkTinftyDelta}). One only needs to look at their cohomology classes: $\sigma_*([\widehat \Delta] \wedge \omega^{k-1})$ belongs to the class of $[\Delta]$ in $X^2$,  hence $ \langle \widehat T_\infty \wedge [\widehat \Delta], \widehat \omega^{k-1}\rangle$ is equal to the cup product of $\{T_\infty\}$ and $\{\Delta\}$ which is the mass of $T^+ \wedge T^{-}.$   Thus (\ref{eq-Pnngafpeirodic}) is proved in this case. 

It remains to treat the case of surfaces. Results from \cite{Diller-Farve,Diller-Guedj-DujardinII,DVT_growth_periodic} show that there exist closed positive currents $T^+, T^-$ on $X$ such that $d_1^{-n} [\Gamma_n]$ converges to $T^+ \otimes T^-$ and $T^+$ has no mass on proper analytic subsets of $X,$ see \cite[Pro. 3.3]{DVT_growth_periodic}. Moreover, by \cite[Cor. 2.4]{DVT_growth_periodic}, the h-dimension of the total tangent class of  $T^+ \otimes T^-$ along $\Delta$ is $0.$ Thus we can apply  Proposition \ref{pro_equidistrG} to get the desired assertion in this case. This finishes the proof.
\endproof

Now we give some more dynamical properties of meromorphic self-maps on manifolds in $\cali{G}.$ 

\begin{corollary} \label{cor_khifsurjective}  Let $X \in \cali{G}$ and $f$ a surjective holomorphic self-map of $X$ with dominant topological degree. Then the following properties hold: 

$(i)$ every Lyapunov exponent of the equilibrium measure $\mu_f$ of $f$ is at least $\frac{1}{2}\log(d_k/d_{k-1}),$ 

$(ii)$ the isolated periodic points of $f$ is equidistributed with respect to $\mu_f:$
\begin{align}\label{eq_limitmuperiodic}
\mu_n:= \frac{1}{P_n} \sum_{x} \nu_x \delta_x \to \mu_f,
\end{align}
where the sum is taken over all isolated periodic points of $f$ of period $n$, $\nu_x$ is the multiplicity of $x$ and $\delta_x$ is the Dirac mass at $x,$

$(iii)$ there exists a totally invariant (possibly empty) proper analytic subset $\mathcal{E}$ of $f,$ \emph{i.e,} $f^{-1}(\mathcal{E})=\mathcal{E}$ such that  $d_k^{-n}(f^n)^*\delta_a \to \mu_f$  if and only  if $a \not \in \mathcal{E}.$   
\end{corollary}

In the K\"ahler case, Corollary \ref{cor_khifsurjective} is already known, see \cite{DS_book,DNT_equi,DinhSibony_allure,Guedj,Briend_Duval_periodic,Duval_Briend_entropy}.   In order to prove Corollary \ref{cor_khifsurjective} we will need the following result which is more or less a trivial extension of \cite[Le. 4.7]{DNV}. 

\begin{lemma} \label{le_BClightramified} Let $X \in \cali{G}$ and $f$ a surjective holomorphic self-map of $X.$ Then the fibers of $f$ are finite, in other words, $f$ is a ramified covering. Moreover  for every constant $\epsilon>0$ and  every irreducible analytic subset $Y$ of dimension $q$ of $X$ with  $f(Y) \subset Y,$ then $f(Y)=Y$ and  the topological degree of $f^n|_Y$ is $\le  \vol(Y)^{-1} \big(d_q(f)+\epsilon\big)^n$ for  $n \ge n_\epsilon$ big enough.  
\end{lemma}

\proof We will argue exactly as in the proof of \cite[Le. 4.7]{DNV} with the Bott-Chern cohomology in place of the de Rham cohomology. Let $\omega$ be a Hermitian metric on $X$ with $\ddc \omega^j=0$ for $1 \le j \le k-1,$ where $k=\dim X.$ The form $\omega^j$ for $1 \le j \le k-1$ induces a class in the Aeppli cohomology of $X,$ see \cite{Daniele_book} for the definition of the Aeppli cohomology. 
 
 Consider the Bott-Chern cohomology $H^{*,*}_{BC}(X)$ of $X$ which is defined by  $H^{*,*}_{BC}(X):=  \Ker d/ \Im \ddc.$ This cohomology $H^{*,*}_{BC}(X)$ for currents or forms is the same and  its dimension is finite because $X$ is compact, see \cite{Daniele_book} for a proof.  Observe that $f^*, f_*$ induce naturally linear endomorphisms on $H^{*,*}_{BC}(X).$ Since $f_* f^* \alpha= d_k \alpha$ for every smooth form $\alpha,$ the map  $f_*  f^*$ acting on the Bott-Chern cohomology is just the multiplication by $d_k.$ Hence, $f_*$ is invertible.  

Suppose that there is $x \in X$ for which $Y:= f^{-1}(x)$ is of a strictly positive dimension $q$. Denote by $\{Y\}_{BC}$ the class of $[Y]$ in $H^{k-q,k-q}_{BC}(X).$ Using $\int_Y \omega^q >0,$ we see that $\{Y\}_{BC}$ is nonzero because of  the duality between the Bott-Chern cohomology and the Aeppli cohomology.  On the other hand, $f_* [Y]=0$ because $f(Y)=\{x\}$ of dimension $0.$ It follows that $f_* (\{Y\}_{BC})=0.$ This is a contradiction because $f_*$ is invertible.  We conclude that the fibers of $f$ are finite. 

Let $Y$ be an irreducible analytic subset of $X$ with $f(Y)\subset Y.$ Since the fiber of $f$ is finite, $f(Y)$ is of the same dimension $q$. Thus $f(Y)=Y.$ We have $f^n(Y)=Y$ for every $n.$   Let $\delta$ be the topological degree of $f^n|_Y.$ Observe $(f^n)_* [Y]= \delta [Y]$ because $Y$ is irreducible.  Let  $\beta_1, \ldots,\beta_m$ be closed forms on $X$ such that their Bott-Chern cohomology classes form a basis of $H^{k-q,k-q}_{BC}(X).$    Let $\Phi$ be a closed form in the  Bott-Chern cohomology class of $[Y].$ We can write $\Phi= \sum_{j=1}^m a_j \beta_j+ \ddc \Phi'$ for some $a_j \in \C$ and some smooth form $\Phi'.$  Consequently, 
\begin{align}\label{ine_massfsaoY}
\| (f^n)_* [Y] \| = \langle (f^n)_* [Y], \omega^q \rangle= \langle (f^n)_* \Phi, \omega^q \rangle \lesssim  \sum_{j=1}^m | \langle (f^n)_* \beta_j, \omega^q \rangle | \lesssim (d_q+\epsilon)^n,
\end{align}     
 for  $n\ge n_\epsilon$ big enough. We deduce that  $\delta \| [Y]\| \le (d_q+ \epsilon)^n.$  This finishes the proof. 
\endproof

\begin{proof}[Proof of Corollary \ref{cor_khifsurjective}]
We first  prove $\mu_n \to \mu_f.$  We just follow the usual idea to construct good inverse branches of $f^n$.  Let $Y$ be the set of critical values of $f.$  The fiber $f^{-1}(x)$ has exactly $d_k$ points for $x \in X \backslash Y.$  The set $Y$ is the image by $f$ of the critical set of $f$ which is a hypersurface.  By  Lemma \ref{le_BClightramified}, $Y$ is also a hypersurface.   

Let $\epsilon$ be a small positive constant for which $d_k > d_j+\epsilon$ for $0 \le j \le k-1$. By (\ref{ine_massfsaoY}) and the fact that $d_k> d_{k-1},$ we see that 
$$R:= \sum_{n\ge 0} d_k^{-n} f^n_* [Y]$$
is a well-defined closed positive $(1,1)$-current on $X.$   This current $R$ is called the ramification current of $f$. Put     $E_1:=\{x\in X: \nu(R,x) \ge 1\}$ which is an analytic subset of $X$ by  Siu's semi-continuity theorem.    Let $\mathcal{E}$ be  the set of $x \in X$ for which $f^{-n}(x) \in E_1$ for every $n\in \N.$   

Arguing exactly as in the proof of  \cite[Pro. 1.51, Th. 1.45]{DS_book}, we see that Property $(iii)$ holds and  $\mathcal{E}$ is totally invariant and maximal in the sense that for every proper analytic subset  $E$ of $X$ with $f^{-s}(E) \subset E$ for some $s\ge 1$ then $E \subset \mathcal{E}.$ We also have that there are at most a finitely many analytic sets in $X$ which are totally invariant. We only need to note that the proofs presented there only used the K\"ahler form to construct $R$ and estimate the topological degree of $f^n|_Y;$ the other arguments hold without the presence of a K\"ahler form.  By the same reason, we obtain the lower bound for the Lyapunov exponents of $\mu_f$ and the equidistribution of isolated periodic points of $f$ as in \cite[Th. 1.57, Th. 1.120]{DS_book}.  Remark that although we don't know  whether $\mu_f$ is ergodic, this issue doesn't affect arguments in  \cite[Th. 1.120]{DS_book}.  
The proof is finished.
\end{proof}

\begin{remark} \label{reuniquenssentropy} 
Consider now $X$ is of dimension $2$ and $f$ a meromorphic self-map of $X$ with dominant topological degree.  By Theorems \ref{th_dominant} and  \ref{th_main2phay}, $\mu_f$ is an invariant measure of maximal entropy.  We can show that $\mu_f$ is a unique measure of maximal entropy  by using Gauduchon's metric instead of a K\"ahler form, arguments in the proof of Theorem \ref{th_main2phay} and   repeating arguments in K\"ahler case \cite[Th. 1.118]{DS_book} or \cite[Th. 2]{Duval_Briend_entropy}. 
\end{remark}

\bibliography{biblio_family_MA,biblio_Viet_papers}
\bibliographystyle{siam}

\end{document}